\documentclass[preprint]{elsarticle}
\usepackage{cmap} 
           \usepackage[T2A]{fontenc}   
           \usepackage[utf8]{inputenc}
           \usepackage[english]{babel}
\usepackage{amsmath}
\usepackage{amsfonts}
\usepackage{amssymb}
\usepackage{amsthm}
\usepackage[pdftex,unicode]{hyperref}
\usepackage{amscd}
\usepackage[all,cmtip]{xy}
\usepackage{mathtools}
\usepackage{comment}
\usepackage{tikz-cd}

  \def\R{\mathbb R} 

\def\uu{\Upsilon}
\def\vu{\Psi}
\def\UU{\mathfrak{U}}

\def\om{\omega}
\def\Om{\Omega}
\def\aut#1{\operatorname{Aut}(#1)}
\def\supp{\operatorname{supp}}

\def\be{\beta}

\def\al{\alpha}
\def\ga{\gamma}

\def\ph{\varphi}
\def\de{\delta}
\def\ka{\kappa}

\def\lll{\Lambda}

\def\ggg{\Gamma}
\def\gbm{\ggg^{\scriptscriptstyle{BM}}}
\def\god{\ggg^{\scriptscriptstyle{OD}}}

\def\ode{{\wt{OD}}}
\def\bme{{\wt{BM}}}

\def\nfs/{NSF}
\def\cdp/{CDP}
\def\cdpz/{CDP${}_0$}

\def\cW{\mathcal{W}}
\def\cN{\mathcal{N}}

\def\cG{\mathcal{G}}

\def\ll{{\mathfrak{L}}}
\def\uu{\Upsilon}

\def\ww{{\mathfrak{W}}}
\def\vv{{\mathfrak{V}}}
\def\ppi{\varpi}

\def\vvbm{\vv_{BM}}
\def\vvbmi{\vvbm^*}
\def\vvr{\vv_{R}}

\def\ll{{\mathrm{L}}}

\def\pn{{\mathcal{N}}}

\def\sq#1#2{(#1)_{#2}}
\def\sqn#1{\sq{#1}{n\in\om}}
\def\sqnn#1{\sqn{#1_n}}

\def\cl#1{\overline{#1}}

\def\Int{\operatorname{Int}}

\def\bt{\operatorname{\beta}}

\def\ltu{\operatorname{\cl{lt}}}
\def\oms{\om^*}

\def\limp#1{{\textstyle \lim_{#1}}}

\def\es{\varnothing}
\def\Tau{{\mathcal T}}

\def\nom{{n\in\om}}

\def\sset#1{\{#1\}}
\def\fset#1{\{#1\}}
\def\set#1{\bbset#1\eeset}
\def\bbset#1:#2\eeset{\{#1\,:\,#2\}}

\def\bbsett#1:#2\eesett{\{#1\,:\,\text{#2}\}}

\def\iset#1{\ibbset#1\ieeset}
\def\ibbset#1:#2\ieeset{(#1)_{#2}}

\def\tp{\Tau}
\def\tps{\tp^*}

\def\ppi#1{\Pi(#1)}
\def\cP{{\mathcal P}}
\def\cB{{\mathcal B}}
\def\gP{{\mathfrak{P}}}

\def\pbase#1{\gP\left[#1\right]}

\def\cB{{\mathcal B}}
\def\cF{{\mathcal F}}

\def\cU{{\mathcal U}}
\def\cO{{\mathcal O}}

\def\cD{{\mathcal D}}

\def\Exp#1{\operatorname{Exp}(#1)}
\def\Expne#1{\operatorname{{Exp}}_*(#1)}

\def\gi{{\mathfrak{g}}}

\def\B{\mathbb D}
\def\til{\tilde}

\def\wt#1{\widetilde{#1}}

\def\ts{{\til{s}}}

\def\eqdef{\coloneqq}
\def\Ss{{\mathcal S}}

\def\cX{{\mathcal X}}
\def\cV{{\mathcal V}}

\def\diag{\mathop{\bigtriangleup}}

\def\St{\operatorname{St}}
\def\st{\operatorname{st}}

\def\fun{{}^\frown}

\def\OR{\bigvee}

\newcommand\restrA[2]{{
  \left.\kern-\nulldelimiterspace 
  #1 
  \vphantom{\big|} 
  \right|_{#2} 
  }}

\newcommand\restrB[2]{\ensuremath{\left.#1\right|_{#2}}}

\def\restr#1#2{\restrB{#1}{#2}}

\def\pwr#1_#2{#1^{[#2]}}

\def\D{\Delta}
\def\term#1{{\it #1}}

\def\alp{{\al\in P}}

\def\dddm#1(#2){N_{#1}(#2)}
\def\dddb#1(#2){B_{#1}(#2)}

\def\et(#1){}

\def\bitem#1,#2.{ $#2\nrightarrow #1$:\ }
\long\def\edemo#1\endedemo{{\rm #1}}

\def\myparagraph#1{\par\noindent{\emph{#1}}\,}

\newtheorem{assertion}{Assertion}
\newtheorem{proposition}{Proposition}
\newtheorem{theorem}{Theorem}
\newtheorem*{theorem*}{Theorem}

\newtheorem*{lemma*}{Lemma}

\newtheorem{example}{Example}

\theoremstyle{definition}
\newtheorem{definition}{Definition}

\newtheorem{problem}{Problem}
\theoremstyle{remark}
\newtheorem{note}{Remark}
\newtheorem*{note*}{Remark}

\begin{document}
\begin{frontmatter}

\title{Classes of Baire spaces defined by topological games}

\author{Evgenii Reznichenko} 
\ead{erezn@inbox.ru}

\address{Department of General Topology and Geometry, Mechanics and  Mathematics Faculty, 
M.~V.~Lomonosov Moscow State University, Leninskie Gory 1, Moscow, 199991 Russia}

\begin{abstract}
The article studies topological games that arise in the study of the continuity of operations in groups with topology, such as paratopological and semitopological groups. These games are modifications of the Banach--Mazur game.

Given a two-player game $G(X)$ of the Banach--Mazur type, we define $\Gamma^G$-Baire, $\Gamma^G$-nonmeager and $\Gamma^G$-spaces. A space $X$ is a $\Gamma^G$-Baire if the second player does not have a winning strategy in $G(X)$. The classes of $\Gamma^G$-nonmeager spaces and $\Gamma^G$-spaces are defined similarly, with the help of modifications of the game $G(X)$.

For the games under consideration, equivalent games are found, which facilitates studying the relationship between the resulting classes of spaces and determining which spaces belong to these classes. For this purpose, we introduce a modification of the Banach--Mazur game with four players.

Results of this paper find application in the study the continuity of operations in groups with topology.
\end{abstract}
\begin{keyword}
Baire space
  \sep
nonmeager space
\sep
topological games
\sep
classes of Baire spaces
\sep 
\MSC[2010] 54B10 \sep 54C30 \sep 54C05 \sep 54C20 
\end{keyword}

\end{frontmatter}

\section{Introduction}

A space $X$ is called \term{Baire} (\term{nonmeager}) if for any family $(U_n)_n$ of open dense subsets $X$ the intersection $G=\bigcap_n U_n$ is dense in $X$ (nonempty).

Baire spaces play an important role in mathematics. Particularly strong results have been obtained in the class of metric spaces.
We note the following two results, in which, in addition to being Baire, an important role is played by metrizability.
\begin{enumerate}
\item
If $X$, $Y$ and $Z$ are metric spaces,
$X$ is a Baire space and a function $f: X\times Y\to Z$ is separately continuous, then $f$ has points of continuity \cite{baire1899,namioka1974,christensen1981}.
\item
If $G$ is a metric Baire group with separately continuous multiplication, then $G$ is a topological group \cite{arh-rezn2005,cdp2010}.
\end{enumerate}

To extend these results from metric spaces to larger classes of spaces, topological games are widely used. An important role in applications of the Baire property is played by topological games that are modifications of the Banach--Mazur game \cite{Telegarsky1987,Revalski2004}, with the help of which a characterization of Baire spaces, the Banach--Oxtoby theorem,  was proved (see the Theorem \ref{tbmb-bm}): a space $X$ is Baire if and only if the second player $\be$ has no winning strategy in the game $BM(X)$.

A standard scheme for extending results of the first and second types from metric Baire spaces to larger classes  is as follows.
\begin{enumerate}
\item
A modification $\ggg(X)$ of the game $BM(X)$ is defined so that, for the class $\cB$ of spaces on which the player $\be$ in the game $\ggg(X)$ does not have a winning strategy, 
theorems known for metrizable Baire spaces remain valid.
\item
Spaces from the class $\cB$ are found. As a rule, these are Baire spaces from some 'traditional' class of spaces $\cP$.  Then, theorems that: if $X\in \cP$ is a Baire space, then $X\in\cB $ are proved.
\end{enumerate}

In this paper, the  class $\cB$ is  a subclass of the class of Baire spaces obtained by using a modification $G(X)$ of the Banach--Mazur game $BM(X)$. We refer to spaces in this class as $\ggg^B$-Baire spaces; see Section \ref{sec-bmb}.
A space $X$ is $\Gamma^B$-Baire if player $\beta$ does not have a winning strategy in $G(X)$.
The classes of $\Gamma^G$-nonmeager spaces and $\Gamma^G$-spaces are defined similarly, with the help of modifications of the game $G(X)$. 

If $X$ is a Baire space and $X$ is a $\ggg^G$-space, then $X$ is a $\ggg^G$-Baire space (Proposition \ref{pbms1}).
Exploring $\ggg^G$-spaces is much easier than $\ggg^G$-Baire spaces. Some of the $\ggg^G$-spaces are described in Theorem \ref{tbms2}. Proposition \ref{pbms1} and Theorem \ref{tbms2} allow us to find $\ggg^G$-Baire spaces. The author does not know if there is a $\ggg^G$-Baire space that is not a $\ggg^G$-space; see Problem \ref{pqe2} (1).

For the games under consideration, equivalent games are found, which facilitates studying the relationship between the resulting classes of spaces and determining which spaces belong to these classes. For this purpose, we introduce a modification of the Banach--Mazur game with four players; see Section \ref{sec-bme}.

The concept of a Baire space is closely related to the concept of a nonmeager space.
\begin{theorem}\label{tintro1}
\begin{itemize}
\item[\rm(1)]
Baire spaces are nonmeager.
\item[\rm(2)]
An open subset of a Baire space is a Baire space.
\item[\rm(3)]
A space $X$ is  Baire if and only if every open subspace of $X$ is a nonmeager space.
\item[\rm(4)]
A space $X$ is nonmeager if and only if there exists an open nonempty Baire subspace of $X$.
\item[\rm(5)]
A homogeneous nonmeager space is a Baire space.
\end{itemize}
\end{theorem}
In the article, generalizations of Baire and nonmeager spaces are constructed in parallel, and the relationships (1)--(5) between Baire and nonmeager spaces are checked.

In \cite{rezn2008,rezn2022-2008} the $\D$-Baire property was found, which implies the continuity of operations in groups. 
The $\D$-Baire property is defined with the help of semineighborhoods of the diagonal. 
Paper \cite{rezn2022gbd} also contains properties of Baire type, which are defined by using semineighborhoods of the diagonal. 
It establishes a relationship between the generalizations of the Baire property obtained with the help of topological games in this paper and 
those 
obtained with the help of 
semineighborhoods of the diagonal.

The results of this paper are used in \cite{rezn2022-1} to study the continuity of group operations in right-topological groups.

\section{Definitions and notation}

The sign $\eqdef$ will be used for equality by definition.

\subsection{Definitions and notation from set theory}

The family of all subsets of a set $X$ is denoted by $\Exp X$.
The family of all nonempty subsets of a set $X$ is denoted by $\Expne X$: $\Expne X\eqdef\Exp X \setminus \fset{\es}$.

If $B$ is a subset of a set $A$, then we denote by $B^c=A\setminus B$ the complement to $A$. We use this notation in situations where it is clear from the context which set $A$ is meant.

An \term{indexed set} $x=\iset{x_\al:\al\in A}$ is a function on $A$ such that $x(\al)=x_\al$ for $\al\in A$.
If the elements of an indexed set $\cX=\iset{X_\al: \al\in A}$ are themselves sets, then $\cX$ is also called an \term{indexed family of sets}; $\cX$ is a function on $A $: $\cX(\al)=X_\al$ for $\alp$.
For a nonempty $B\subset A$, we denote
\[
\pwr\cX_B\eqdef\prod_{\al\in B}X_\al=\set{\iset{x_\al:\al\in B}:x_\al\in X_\al\text{ for all }\al\in B}
\]
The projection from $\pwr\cX_B$ onto $X_\al$ will be denoted by $\pi_\al$.
We assume that $\pwr\cX_B=\sset\es$ if $B$ is the empty set:
\[
\pwr\cX_{\es}\eqdef \sset\es.
\]
The Cartesian product $\pwr\cX_B$ is the set of functions $f$ defined on the set $B$ such that $f(\al)\in \cX(\al)$ for all $\al \in B$.
We denote
\[
\prod\cX\eqdef\prod_{\al\in A}X_\al=\pwr \cX_A.
\]

Let $B\cap C=\es$.
As is customary in set theory, we identify a function with its graph.
If $x\in \pwr\cX_B$ and $y\in\pwr\cX_C$, then $z=x\cup y$ is the function defined by 
\[
z\in \pwr\cX_{B\cup C},\
x=\restr zB\text{ and }y=\restr zC
\]
Let us introduce a special notation for $x\cup y$ when $x$ and $y$ are functions:
\[
x\fun y \eqdef x\cup y.
\]
Functions with a finite domain are sets of the form
\[
f=\{(x_1,y_1),(x_2,y_2),\dots ,(x_n,y_n)\},
\]
$f(x_i)=y_i$ for $i=1,2,\dots ,n$.
We will use the notation
\[
\{x_1\to y_1, x_2\to y_2,\dots ,x_n\to y_n\} \eqdef
\{(x_1,y_1),(x_2,y_2),\dots ,(x_n,y_n)\}.
\]
In particular,
\begin{align*}
\{\al\to a\}&=\sset{(\al,a)},
&
\{\al\to a,\be\to b\}&=\sset{(\al,a),(\be,a))}.
\end{align*}

\subsection{Definitions and notation from topology}

We denote by $\aut X$ the set of all homeomorphisms of the space $X$ onto itself.

A subset $M$ of a topological space $X$ is called \term{locally dense}, or \term{nearly open}, or \term{preopen} if $M\subset \Int \cl M$.

Let $M\subset X$. If $M$ is the union of a countable number of nowhere dense sets, then $M$ is called a \term{meager} set\et(meager set). Nonmeager sets are called \term{sets of the second Baire category}.
A subset of $M$ is said to be \term{residual}, or \term{comeager}, if $X\setminus M$ is a meager set.

A space $X$ is called \term{a space of the first Baire category}, or a \term{meager spaces}, if 
the set $X$ is of the first Baire category in the space $X$.
 A space $X$ is called \term{a space of the second Baire category}, or \term{nonmeager spaces}, if $X$ is not a meager space. A space in which every residual set is dense is called a \term{Baire space}.
A space is nonmeager if and only if some open subspace is a Baire space.

A family $\nu$ of nonempty subsets of $X$ is called a \term{$\pi$-net} if for any open nonempty $U\subset X$ there exists an $M\in\nu$ such that $M\subset U$.

A $\pi$-network consisting of open sets is called a \term{$\pi$-base}.

A subset $U\subset X$ is said to be \term{regular open} if $U=\Int{\cl U}$.

A space $X$ is called \term{quasi-regular} if for every nonempty open $U\subset X$ there exists a nonempty open $V\subset X$ such that $\cl V\subset U$.

A space $X$ is \term{semiregular} if $X$ has a base consisting of regular open sets.

A space $X$ is called \term{$\pi$-semiregular} \cite{Ravsky2001} (or nearly regular \cite{Ameen2021}) if $X$ has a $\pi$-base consisting of regular open sets.

For a cardinal $\tau$, a set $G\subset X$ is called a \term{set of type $G_\tau$} if $G$ is an intersection of $\tau$ open sets. A space $X$ is called \term{an absolute $G_\tau$ space} if $X$ is of type $G_\tau$ in some compact extension.

A space \term{$X$ is regular at a point $x\in X$} if for any neighborhood $U$ of the point $x$ there exists a neighborhood $V\ni x$ such that $\cl V\subset U$.

A space \term{$X$ is semiregular at a point $x\in X$} if there is a base at the point $x$ consisting of regular open sets.

A space $X$ is \term{feebly compact} if any locally finite family of open sets is finite.

For $\ga\subset \Exp X$ and $x\in X$ we denote
\begin{align*} 
\St(x,\ga)&\eqdef\set{U\in\ga: x\in \ga},
&
\st(x,\ga)&\eqdef\bigcup\St(x,\ga).
\end{align*}

A space $X$ is called \term{developable}\et(developable) if there exists a sequence of open covers $(\ga_n)_{n\in\om}$ such that for any $x\in X$ the family $\st(x,\ga_n)$ is a base at the point $x$.

A family $\cB$ of open nonempty sets in $X$ is called \term{an outer base of $M\subset X$} if $M\subset U$ for each $U\in\cB$ and for each open $W\supset M$ there exists a $U\in \cB$ such that $M\subset U\subset W$.

If $\sqnn M$ is a sequence of subsets of a space $X$, then the set
\begin{align*}
\ltu_{n\in\om} M_n = \{x\in X:\ &|\set{n\in\om: U\cap M_n\neq\es}|=\om
\\
&\text{ for any neighborhood }U\text{ of }x\}
\end{align*}
is called the \term{upper limit of the sequence of sets $\sqnn M$}.

If $\sqnn x$ is a sequence of points in the space $X$, then we denote
\begin{align*}
\ltu_{n\in\om} x_n & \eqdef \ltu_{n\in\om} \sset{x_n}.
\end{align*}

We denote by $\bt \om$ the space of ultrafilters on $\om$, the Stone-\v{C}ech extension of the discrete space $\om$. We denote by $\oms=\bt\om\setminus \om$ --- the set of nonprincipal ultrafilters.

Let $\sqnn x$ be a sequence of points in a space $X$, and let $p\in \oms$ be a nonprincipal ultrafilter. A point $x\in X$ is called the \term{$p$-limit} of a sequence $\sqnn x$ if $\set{n\in\om: x_n\in U}\in p$ for any neighborhood $U$ of $x$. We will write $x=\limp p x_n=\limp p \sqnn x$ for the $p$-limit $x$.

\section{Modifications of the Banach-Mazur game}\label{sec-bmg}

In this section, we use topological games; the basic concepts and terminology for them can be found in \cite{Oxtoby1957,tel-gal-1986,arh2010,rezn2008,rezn2022-2008}. 
A precise definition of a game  is given in Section \ref{sec-bme}.
In this section, we assume that there are two players, $\al$ and $\be$.
Let $G$ be a game in which a player $\ka\in\sset{\al,\be}$ has a winning strategy.
Let us call this game
 $G$ \term{$\ka$-favorable}. If there is no such strategy, then $G$ is a \term{$\ka$-unfavorable} game.

If the definition of the game $G$ depends on only one parameter, namely, some space $X$, that is, $G=\ggg(X)$, then we say that the \term{space $X$ is $(\ka,\ggg)$-favorable} if the game $\ggg(X)$ is $\ka$-favorable and  the space $X$ is \term{$(\ka,\ggg)$-unfavorable} if the game is $\ggg(X)$  is $\ka$-unfavorable.

Let $G_1$ and $G_2$ be two games with players $\al$ and $\be$. We say that the games $G_1$ and $G_2$ are \term{equivalent} if the game $G_1$ is $\ka$-favorable if and only if the game $G_2$ is $\ka$-favorable for all $\ka\in\sset {\al,\be}$. We will write $G_1\sim G_2$ for equivalent games.

Let $(X,\tp)$ be a space. We set $\tps=\tp\setminus\sset\es$ and denote
\begin{align*}
\vv(X)&\eqdef \set{\sqnn V\in {\tps}^\om: V_{n+1}\subset V_n\text{ for }\nom}.
\end{align*}
We put
\begin{align*}
\UU(X)&\eqdef \set{\uu\in (\Expne{\tps})^{\tps}: \uu(U)\text{ is a $\pi$-base in }U\in\tps}.
\end{align*}
Let $\cP$ be some $\pi$-base of the space $X$. Let us define $\uu_t(X),\uu_r(X),\uu_p(X,\cP)\in (\Exp{\tps})^{\tps}$ as follows. For $U\in\tps$ we put
\begin{align*}
\uu_t(X)(U)&=\set{V\in\tps: V\subset U},
&
\uu_p(X,\cP)(U)&=\set{V\in\cP: V\subset U},
\\
\uu_r(X)(U)&=\set{V\in\tps: \cl V\subset U},
&
\uu_{pr}(X,\cP)(U)&=\set{V\in\cP: \cl V\subset U}.
\end{align*}
Obviously, $\uu_t(X),\uu_p(X,\cP)\in\UU(X)$, and if the space $X$ is quasiregular, then $\uu_r(X),\uu_{pr}(X,\cP)\in\UU(X)$.

\paragraph{Games $BM(X,\cV;\uu,\vu)$ and $MB(X,\cV;\uu,\vu)$}
Let $\cV\subset \vv(X)$ and $\uu,\vu\in\UU(X)$.
There are two players, $\al$ and $\be$. These games differ in the first move of a player $\al$.
On the first move, player $\al$ chooses $U_0=X$ in the game $BM(X,\cV;\uu,\vu)$ and $U_0\in\uu(X)$ in the game $MB(X,\cV;\uu,\vu)$. Player $\be$ chooses $V_0\in\vu(U_0)$. On the $n$th move, $\al$ chooses $U_n\in\uu(V_{n-1})$ and $\be$ chooses $V_n\in\vu(U_n)$. After a countable number of moves, the winner is determined: player $\al$ wins if $\sqnn V\in\cV$.

We put
\begin{align*}
BM(X,\cV) &\eqdef BM(X,\cV;\uu_t(X),\uu_t(X)),
\\
MB(X,\cV) &\eqdef MB(X,\cV;\uu_t(X),\uu_t(X)).
\end{align*}

\begin{definition}
Let $X$ be a space. A family $\cV\subset \vv(X)$ is called \term{monolithic}\et(monolithic) if the following condition is satisfied:
\begin{itemize}
\item[{}] Let $\sqnn U\in \vv(X)$ and $\sqnn V\in \cV$. If $U_{n+1}\subset V_n \subset U_n$ for $\nom$, then $\sqnn U\in \cV$.
\end{itemize}
\end{definition}
\begin{note}
The paper
\cite{arh2010} introduced a similar concept of a stable family, and \cite{tel-gal-1986} introduced the concept of a monotone family.
A monotone family is stable and monolithic.
The reason for introducing a new class of monolithic families is that $\vvbm(X)$ (see~Section \ref{sec-bmb}) is a monolithic, but not  monotone or stable family.
\end{note}

\begin{proposition}\label{pbmg1}
Let $X$ be a space, $\cV\subset \vv(X)$, $\uu,\vu\in\UU(X)$.
If $\cV$ is a monolithic family, then
\begin{align*}
BM(X,\cV)&\sim BM(X,\cV;\uu,\vu),
&
MB(X,\cV)&\sim MB(X,\cV;\uu,\vu).
\end{align*}
\end{proposition}
\begin{note}
Proposition \ref{pbmg1} will be proved after Proposition \ref{pbmg3}. Proposition \ref{pbmg1} allows one to pass from the game $BM(X,\cV)$ to the game $BM(X,\cV;\uu,\vu)$ in which the players choose open sets not arbitrarily, but in some special way, for example, from some convenient $\pi$-base.
\end{note}

Let
\begin{align*}
\ww(X)\eqdef\{ \sqn{V_n,M_n}\in (\tps\times\Expne X)^\om:&\ M_{n+1}\subset V_n\text{ and }
\\
&V_{n+1}\subset V_n
\text{ for }\nom
\}.
\end{align*}


\paragraph{Games $OD(X,\pn,\cW;\uu,\vu)$ and $DO(X,\pn,\cW;\uu,\vu)$}
Let $\pn$ be a $\pi$-net of $X$. Take $\cW\subset \ww(X)$ and $\uu,\vu\in\UU(X)$.
There are two players, $\al$ and $\be$. These games are distinguished by the first move of player $\al$.
On the first move, $\al$ chooses $U_0=X$ in $OD(X,\pn,\cW;\uu,\vu)$ and $U_0\in\uu(X)$ in $DO( X,\pn,\cW;\uu,\vu)$. Player $\be$ chooses $V_0\in\vu(U_0)$ and $M_0\in\pn$, $M_0\subset U_0$. On $n$th move $\al$ chooses $U_n\in\uu(V_{n-1})$ and $\be$ chooses $V_n\in\vu(U_n)$ and $M_n\in\pn $, $M_n\subset U_n$. After a countable number of moves, the winner is determined: player $\al$ wins if $\sqn{V_n,M_n}\in\cW$.

We put
\begin{align*}
OD(X,\pn,\cW) &\eqdef OD(X,\pn,\cW;\uu_t(X),\uu_t(X)),
\\
DO(X,\pn,\cW) &\eqdef DO(X,\pn,\cW;\uu_t(X),\uu_t(X)).
\end{align*}

\begin{definition}
Let $X$ be a space. A family $\cW\subset\ww(X)$ is called \term{monolithic}\et(monolithic) if the following condition is satisfied:
\begin{itemize}
\item[{}] Let $\sqnn U\in \vv(X)$ and $\sqn {V_n,M_n}\in \cW$. If $U_{n+1}\subset V_n \subset U_n$ for $\nom$, then $\sqn {U_n,M_n}\in\cW$.
\end{itemize}
\end{definition}

\begin{proposition}\label{pbmg2}
Let $X$ be a space, $\pn$ be a $\pi$-net of $X$, $\cW\subset \ww(X)$, $\uu,\vu\in\UU(X)$.
If $\cW$ is a monolithic family, then
\begin{align*}
OD(X,\pn,\cW)&\sim OD(X,\pn,\cW;\uu,\vu),
\\
DO(X,\pn,\cW)&\sim DO(X,\pn,\cW;\uu,\vu).
\end{align*}
\end{proposition}
Proposition \ref{pbmg2} will be proved later (see Proposition \ref{pbme2}).

Let $\cV\subset \vv(X)$. We put
\begin{align*}
\ww_v(X,\cV)\eqdef\{ \sqn{V_n,M_n}\in \ww(X):&\ \sqnn V\in\cV \}.
\end{align*}

\begin{proposition}\label{pbmg3}
If $X$ is a space, $\cV\subset \vv(X)$, $\pn$ $\pi$-net of $X$, $\cW=\ww_v(X,\cV)$ and $\uu,\vu\in\UU(X)$, then
\begin{align*}
BM(X,\cV;\uu,\vu)&\sim OD(X,\pn,\cW;\uu,\vu),
&
MB(X,\cV;\uu,\vu)&\sim DO(X,\pn,\cW;\uu,\vu).
\end{align*}
\end{proposition}
\begin{proof}
For $\cW=\ww_v(X,\cV)$ the outcome of the games $OD$ and $DO$ does not depend on the choice of $M_n$, so the strategies from the games $BM$ and $MB$ are suitable for the games $OD$ and $ DO$.
\end{proof}
Proposition \ref{pbmg3} shows that $BM$ ($MB$) games are a special case of $OD$ ($DO$) games.

\begin{proof}[Proof of Proposition \ref{pbmg1}]
Let $\pn$ be some $\pi$-net in the space $X$, and let $\cW=\ww_v(X,\cV)$. From Proposition \ref{pbmg3} it follows that
\begin{align*}
BM(X,\cV;\dots )&\sim OD(X,\pn,\cW;\dots ),
&
MB(X,\cV;\dots )&\sim DO(X,\pn,\cW;\dots ).
\end{align*}
The family $\cW$ is monolithic if and only if the family $\cV$ is monolithic. It remains to apply Proposition \ref{pbmg2}.
\end{proof}


\begin{proposition}\label{pbmg4}
Let $X$ be a space. Suppose that $\cW\subset \ww(X)$, $\uu,\vu\in\UU(X)$, $\pn_1$ and $\pn_2$ are  $\pi$-nets of the space $X$, and the following conditions are met:
\begin{itemize}
\item[\rm (1)] for $M_1\in \pn_1$ there is an $M_2\in \pn_2$ such that $M_2\subset M_1$ and for $M_2\in \pn_2$ there is an  $M_1\in \pn_1$ such that $M_1\subset M_2$;
\item[\rm (2)] if $\sqn{V_n,M_n}\in \cW$ and $M_n\subset L_n\subset V_n$ for $\nom$, then $\sqn{V_n,L_n}\in \cW$.
\end{itemize}
Then
\begin{align*}
OD(X,\pn_1,\cW;\uu,\vu)&\sim OD(X,\pn_2,\cW;\uu,\vu),
\\
DO(X,\pn_1,\cW;\uu,\vu)&\sim DO(X,\pn_2,\cW;\uu,\vu).
\end{align*}
\end{proposition}
\begin{proof} Fix $\ph_1: \pn_1\to \pn_2$ and $\ph_2: \pn_2\to \pn_1$ so $\ph_1(M_1)\subset M_1$ for $M_1\in \pn_1$ and $\ph_2(M_2)\subset M_2$ for $M_2\in \pn_2$.

Suppose that the player $\al$ in the game $G_1=OD(X,\pn_1,\cW;\uu,\vu)$ has a winning strategy $s_1$. Let us describe a winning strategy $s_2$ for $\al$ in the game $G_2=OD(X,\pn_2,\cW;\uu,\vu)$. We put
\[
s_2(U_0,V_0,M_0,\dots ,V_{n-1},M_{n-1})=U_n=s_1(U_0,V_0,\ph_2(M_0),\dots ,V_{n-1 },\ph_2(M_{n-1})).
\]

Let player $\be$ in the game $G_1$ have a winning strategy $s_1$. Then a winning strategy $s_2$ for $\be$ in the game $G_2$ is as follows.  On the $k$th move player $\be$  chooses an open $V_k$ and $L_k\in \pn_1$, $L_k\subset V_k$, $M_k=\ph_1(L_k)$.
We put
\begin{align*}
(V_n,L_n)&=s_1(U_0,V_0,L_0,\dots ,V_{n-1},L_{n-1},U_n),
\\
M_n&=\ph_1(L_n),
\\
s_2(U_0,V_0,M_0,\dots ,V_{n-1},M_{n-1},U_n)&=(V_n,M_n).
\end{align*}

For the game $DO$ the proof is similar.
\end{proof}

\begin{definition}
A strategy of player $\al$ in games $BM,MB,OD,DO$ will be called \term{regular} if $\cl{U_{n+1}}\subset V_n$ for $\nom$.
\end{definition}

\begin{proposition}\label{pbmg5}
Let $X$ be a quasi-regular space, $G$ be one of the games
	$BM(X,\cV)$, 
	$MB(X,\cV)$, 
	$OD(X,\pn,\cW)$, 
	$DO(X,\pn,\cW)$, 
where $\cV\subset \vv (X)$, $\pn$ be a $\pi$-net of $X$, $\cW\subset \ww(X)$, $\cV$ and $\cW$ be monolithic families.
\begin{itemize}
\item[\rm (1)] If $\al$ has a winning strategy in $G$, then there is a winning regular strategy.
\item[\rm (2)] Suppose that player $\be$ has chosen a strategy $s$ in $G$ and player $\al$ has a strategy that outperforms the strategy $s$. Then player $\al$ has a regular strategy that outperforms the strategy $s$.
\end{itemize}
\end{proposition}
\begin{proof}
If $X$ is a quasi-regular space, then $\uu=\uu_r(X)\in\UU(X)$.
Let $\vu=\uu_t(X)\in\UU(X)$.
Then, by virtue of Propositions \ref{pbmg1} and \ref{pbmg2},
\begin{align*}
BM(X,\cV)&\sim BM(X,\cV;\uu,\vu),
&
MB(X,\cV)&\sim MB(X,\cV;\uu,\vu),
\\
OD(X,\pn,\cW)&\sim OD(X,\pn,\cW;\uu,\vu),
&
DO(X,\pn,\cW)&\sim DO(X,\pn,\cW;\uu,\vu).
\end{align*}
\end{proof}

\section{Generalization of Baire and nonmeager spaces through games}\label{sec-bmb}

Let $(X,\tp)$ be a space, $\tps=\tp\setminus\sset\es$.
We denote
\begin{align*}
\vvbm(X)&\eqdef \set{\sqnn V\in \vv(X): \bigcap_{\nom} V_n\neq\es },
\\
\vvbmi(X)&\eqdef \vv(X)\setminus \vvbm(X),
\\
\vvr(X)&\eqdef \set{\sqnn V \in \vv(X): \cl{V_{n+1}}\subset V_n\text{ for }\nom}.
\end{align*}

We put $BM(X)=BM(X,\vvbm(X))$. This is the classical Banach--Mazur game.  We put $MB(X)=MB(X,\vvbm(X))$.

\begin{theorem}[Banach--Oxtoby \cite{Oxtoby1957}; see also \cite{tel-gal-1986,arh2010}] \label{tbmb-bm}
Let $X$ be a space.
\begin{itemize}
\item[\rm(1)]
$X$ is Baire if and only if $BM(X)$ is $\be$-unfavorable;
\item[\rm(2)]
$X$ is nonmeager if and only if $MB(X)$ is $\be$-unfavorable.
\end{itemize}

\end{theorem}

\begin{definition}
Let $X$ be a space, let $\cV\subset \vv(X)$ and let $\cV^*=\cV\cup\vvbmi(X)$. We say that the space $X$ is
\begin{itemize}
\item \term{$\gbm(\cV)$-nonmeager} if $MB(X,\cV)$ is $\be$-unfavorable;
\item \term{$\gbm(\cV)$-Baire} if $BM(X,\cV)$ is $\be$-unfavorable;
\item a \term{$\gbm(\cV)$-space} if $BM(X,\cV^*)$ is $\al$-favorable.
\end{itemize}
\end{definition}

\begin{proposition} \label{pbmb1-1}
Let $X$ be a space and let $\cV_1\subset\cV_2\subset \vv(X)$.
\begin{itemize}
\item[\rm (1)] If $X$ is $\gbm(\cV_1)$-nonmeager, then $X$ is $\gbm(\cV_2)$-nonmeager.
\item[\rm (2)] If $X$ is $\gbm(\cV_1)$-Baire, then $X$ is $\gbm(\cV_2)$-Baire.
\item[\rm (3)] If $X$ is a $\gbm(\cV_1)$-space, then $X$ is a $\gbm(\cV_2)$-space.
\end{itemize}
\end{proposition}
\begin{proof}
In the games $MB(X,\cV_2)$ and $BM(X,\cV_2)$ player $\al$ i uses the strategy from the games $MB(X,\cV_1)$ and $BM(X,\cV_1)$, respectively.
\end{proof}

\begin{proposition} \label{pbmb1}
Let $X$ be a space, and let $\cV\subset \vv(X)$ be a monolithic family.
\begin{itemize}
\item[\rm(1)]
If $X$ is nonmeager and $X$ is a $\gbm(\cV)$-space, then $X$ is $\gbm(\cV)$-nonmeager.
\item[\rm(2)]
If $X$ is a Baire space and $X$ is a $\gbm(\cV)$-space, then $X$ is a $\gbm(\cV)$-Baire space.
\end{itemize}
\end{proposition}

\begin{note}
In \cite{arh2010} Theorem 4.3 was proved, which is similar to \ref{pbmb1}.  The paper \cite{arh2010} considered the game $G_{\cV}$, which differs slightly from $BM(X,\cV)$ by the payoff function: in the game $BM(X,\cV)$ player $\al$ wins if $\sqnn V\in\cV$, and in the game $G_{\cV}$, if $\sqnn U\in\cV$. For $\cV$  used in most of applications, the games $G_{\cV}$ and $BM(X,\cV)$ are equivalent. Below, after Proposition \ref{pbmb3}, we give another proof of Proposition \ref{pbmb1}.
\end{note}

We put
\begin{align*}
\ww_e(X)\eqdef\{ \sqn{V_n,M_n}\in \ww(X):&\ \bigcap_{\nom} V_n=\es \}.
\end{align*}

\begin{definition}
Let $X$ be a space, $\pn$ be a $\pi$-net $X$, $\cW\subset\ww(X)$ and $\cW^*=\cW\cup \ww_e(X)$.
We say that the space $X$
\begin{itemize}
\item \term{$\god(\pn,\cW)$-nonmeager} if $DO(X,\pn,\cW)$ is $\be$-unfavorable;
\item \term{$\god(\pn,\cW)$-Baire} if $OD(X,\pn,\cW)$ is $\be$-unfavorable;
\item a \term{$\god(\pn,\cW)$-space} if $OD(X,\pn,\cW^*)$ is $\al$-favorable.
\end{itemize}
\end{definition}

\begin{proposition} \label{pbmb2-1}
Let $X$ be a space, $\pn_1,\pn_2$ be $\pi$-nets $X$, $\pn_2\subset \pn_1$ and $\cW_1\subset\cW_2\subset \ww(X)$.
\begin{itemize}
\item[\rm (1)] If $X$ is $\god(\pn_1,\cW_1)$-nonmeager, then $X$ is $\god(\pn_2,\cW_2)$-nonmeager.
\item[\rm (2)] If $X$ is $\god(\pn_1,\cW_1)$-Baire, then $X$ is $\god(\pn_2,\cW_2)$-Baire.
\item[\rm (3)] If $X$ is a $\god(\pn_1,\cW_1)$-space, then $\god(\pn_2,\cW_2)$-space.
\end{itemize}
\end{proposition}
\begin{proof}
In the games $OD(X,\pn_2,\cW_2)$ and $DO(X,\pn_2,\cW_2)$ player $\al$  uses the strategy from the games $OD(X,\pn_1,\cW_1)$ and $ DO(X,\pn_1,\cW_1)$, respectively.
\end{proof}

For $\cW\subset \ww(X)$ we set
\begin{align*}
\vv_w(X,\pn,\cW)\eqdef\{ \sqnn V & \in \vv(X) : \text{ if }M_n\subset V_n\text{ and }M_n\in \pn
\\
&\text{ for }\nom\text{, then } \sqn{V_n,M_n}\in\cW
\}.
\end{align*}

\begin{proposition} \label{pbmb2}
Let $X$ be a space, $\pn$ $\pi$-network $X$, $\cW\subset\ww(X)$, and $\cV=\vv_w(X,\pn,\cW)$.
\begin{itemize}
\item[\rm(1)]
If $X$ is $\gbm(\cV)$-nonmeager, then $X$ is $\god(\pn,\cW)$-nonmeager.
\item[\rm(2)]
If $X$ is $\gbm(\cV)$-Baire, then $X$ is $\god(\pn,\cW)$-Baire.
\end{itemize}
\end{proposition}
\begin{proof}
Statement  (1) is equivalent to saying that if $OD(X,\pn,\cW)$ is $\be$-favorable, then $BM(X,\cV)$ is $\be$-favorable. The strategy for $\be$ in the game $BM(X,\cV)$ is that $\be$ chooses $V_n$ according to the winning strategy in the game $OD(X,\pn)$ and $M_n\in \pn$, $M_n\subset V_n$, arbitrarily. Statement (2) can be proved in the same way as (1).
\end{proof}

\begin{proposition} \label{pbmb3}
Let $X$ be a space, $\pn$ be a $\pi$-net $X$, and $\cW\subset\ww(X)$ be a monolithic family.
\begin{itemize}
\item[\rm(1)]
If $X$ is nonmeager and $X$ is a $\god(\pn,\cW)$-space, then $X$ is $\god(\pn,\cW))$-nonmeager.
\item[\rm(2)]
If $X$ is a Baire space and $X$ is a $\god(\pn,\cW)$-space, then $X$ is a $\god(\pn,\cW)$-Baire space.
\end{itemize}
\end{proposition}

\begin{note}
In \cite{rezn2008,rezn2022-2008} (Proposition 3) Proposition \ref{pbmb3} is proved for a few specific $\cW$ and $\pn$, 
but the idea of the proof is also valid for the general case. 
Below we give a proof of Proposition \ref{pbmb3} (see also Proposition \ref{pbme3}).
\end{note}

\begin{proof}[Proof of Proposition \ref{pbmb1}]
Let $\cV^*=\cV\cup\vvbmi(X)$, $\cW=\ww_v(X,\cV)$, and $\cW^*=\cW\cup \ww_e(X)$. Proposition \ref{pbmg3} implies that the games $BM(X,\cV)$, $MB(X,\cV)$ and $BM(X,\cV^*)$ are equivalent to the games $OD(X,\pn,\cW)$, $DO(X,\pn,\cW)$ and $OD(X,\pn,\cW^*)$, respectively. Consequently, the properties of being $\gbm(X,\cV)$-nonmeager, $\gbm(X,\cV)$-Baire, and $\gbm(X,\cV)$-spaces coincide with the properties of being $\god(X, \pn,\cW)$-nonmeager, $\god(X,\pn,\cW)$-Baire, and $\god(X,\pn,\cW)$-spaces, respectively. The fact that $\cV$ is monolithic implies that $\cW$ is monolithic. Now Proposition \ref{pbmb1} follows from Proposition \ref{pbmb3}.
\end{proof}

For $q\in\sset{l,k}$ we define the families $\ww_q(X)\subset \ww(X)$. We say that a  sequence $\sqn{V_n,M_n}\in \ww(X)$ belongs to $\ww_q(X)$ if condition $(\ww_q)$ is met:
\begin{itemize}
\item[$(\ww_l)$] $\ltu_{\nom} M_n \cap \bigcap_{\nom} V_n\neq\es$;
\item[$(\ww_k)$] there is a sequence $\sqnn x$ such that
\begin{itemize}
\item[$(LSQ)$] $x_n\in M_n$ for $\nom$ and for every $p\in\oms$ there is an $x\in \bigcap_{\nom} V_n$ such that $x=\limp p x_n$.
\end{itemize}
\end{itemize}

Note that if $X$ is a regular space, then the condition $(LSQ)$ is equivalent to the condition
\begin{itemize}
\item[]
\begin{itemize}
\item[$(LSQ)'$] 
$x_n\in M_n$ for $\nom$, the subspace $\cl{\set{x_n:\nom}}$ is compact and $\ltu_{\nom} x_n \subset \bigcap_{\nom} V_n$.
\end{itemize}
\end{itemize}

We denote
\begin{align*}
\ll(\sqn{V_n,M_n})=\{\sqnn x\in X^\om:\ & \sqnn x
\text{ satisfies }(LSQ)\}.
\end{align*}
A sequence $\sqn{V_n,M_n}$ is included in $\ww_k(X)$ if and only if $\ll(\sqn{V_n,M_n})\neq\es$.

We put
\begin{align*}
\pn_o(X)&\eqdef \tps, & \pn_p(X)&\eqdef \set{\sset x:x\in X},
\\
\vv_o(X)&\eqdef \vv_w(X,\pn_o(X),\ww_l(X)),
&
\vv_p(X)&\eqdef \vv_w(X,\pn_p(X),\ww_l(X)),
\end{align*}
\begin{align*}
\vv_f(X)\eqdef \{\sqnn U\in\vvbm (X) :\
& \text{for some }x\in \bigcap_{\nom} U_n
\\
&\text {the family }\sqnn U\text { forms a base at the point }x
\},
\\
\vv_k(X)\eqdef \{\sqnn U\in\vvbm (X) :\
& \bigcap_{\nom} U_n=M\text{ is compact and the family }\sqnn U
\\
&\text {is an outer base of the set }M
\}.
\end{align*}

Note that
\begin{itemize}
\item
$\sqnn V \in \vv_o(X)$ if and only if $\sqnn V \in \vv_{BN}(X)$ and for any sequence  $\sqnn M$, $M_n\subset V_n $ for $\nom$ of open nonempty sets we have $\ltu_{\nom} M_n\subset \bigcap_\nom V_n$;
\item
$\sqnn V \in \vv_p(X)$ if and only if $\sqnn V \in \vv_{BN}(X)$ and for any sequence of points $\sqnn x$, $x_n\in V_n$ for $\nom$, we have $\ltu_{\nom} x_n\subset \bigcap_\nom V_n$.
\end{itemize}

The following proposition is easily verified.
\begin{proposition} \label{pbmb4-1}
For any space $X$ the families $\vv_r(X)$ and $\ww_{q}(X)$ are monolithic for $r\in\sset{o,p,f,k}$ and $q\in\sset{l,k} $.
\end{proposition}

For $r\in\sset{o,p,f,k}$ we define the games
\begin{align*}
BM_r(X)&\eqdef BM(X,\cV), & MB_r(X)&\eqdef MB(X,\cV),
\\
BM_r^*(X)&\eqdef BM(X,\cV^*), & MB_r^*(X)&\eqdef MB(X,\cV^*),
\\
\text{where}\quad\quad\quad\quad&&&
\\
\cV&=\vv_r(X), & \cV^*&=\cV\cup \vvbmi(X).
\end{align*}

For $t\in\sset{o,p}$ and $q\in\sset{l,k}$ we define the games
\begin{align*}
OD_{t,q}(X)&\eqdef OD(X,\cN,\cW), & DO_{t,q}(X)&\eqdef DO(X,\cN,\cW),
\\
OD_{t,q}^*(X)&\eqdef OD(X,\cN,\cW^*), & DO_{t,q}^*(X)&\eqdef DO(X,\cN,\cW^*),
\\
\text{where}\hfill\quad\quad&& \cN&=\pn_t(X),
\\
\cW&=\ww_q(X), & \cW^*&=\cW\cup \ww_e(X).
\end{align*}

\begin{definition} \label{dbmb1}
Let $X$ be a space, $r\in\sset{o,p,f,k}$. We say that the space $X$
\begin{itemize}
\item \term{$\gbm_r$-nonmeager} if $X$ is $(\be,MB_r)$-unfavorable;
\item \term{$\gbm_r$-Baire} if $X$ is $(\be,BM_r)$-unfavorable;
\item a \term{$\gbm_r$-space} if $X$ is $(\al,BM_r^*)$-favorable.
\end{itemize}
The class of $\gbm_r$-spaces will be denoted as $\gbm_r$.

Let $t\in\sset{o,p}$ and $q\in\sset{l,k}$.
We say that the space $X$
\begin{itemize}
\item \term{$\god_{t,q}$-nonmeager} if $X$ is $(\be,DO_{t,q})$-unfavorable;
\item \term{$\god_{t,q}$-Baire} if $X$ is $(\be,OD_{t,q})$-unfavorable;
\item a \term{$\god_{t,q}$-space} if $X$ is $(\al,OD_{t,q}^*)$-favorable.
\end{itemize}
The class of $\god_{t,q}$-spaces will be denoted as $\god_{t,q}$.
\end{definition}

\begin{definition}
Let $X$ be a space and let $t\in\sset{o,p,k,f}$. 
We say that a point $x$ is 
\term{$q_t$-point} if there exists a $\sqnn V\in\vv_t(X)$ such that $x\in \bigcap_\nom V_n$.
\end{definition}

Recall that a point $x\in X$ is called a \term{$q$-point} if there exist a $\sqnn V\in\vv_{BM}(X)$ such that $x\in\bigcap_\nom V_n$ and  any sequence $\sqnn x$, $x_n\in V_n$ for $\nom$, accumulates to some point; see \cite{Michael1964}. If $X$ is a regular space, then $x$ is a $q$-point if and only if $x$ is a $q_p$-point. Spaces of point-countable type are precisely spaces in which each point is a $q_k$-point. A point is a $q_f$-point if and only if this point has a countable base.

The following proposition is a direct consequence of the definitions.

\begin{proposition}\label{pbmb4q}
Let $X$ be a space, and let $t\in\sset{o,p,k,f}$. If $X$ is a $\gbm_t$-nonmeager ($\gbm_t$-Baire) space, then there are $q_t$-points in $X$ (the set of $q_t$-points is dense in $X$).
\end{proposition}

\begin{definition}[\cite{rezn2008,rezn2022-2008}]
Let $X$ be a space, and let $Y\subset X$. We call $Y$ \term{$C$-dense} if $\cl Y=X$ and for any countable family $\gamma$ of open subsets of $X$ the family $\gamma$ is locally finite if and only if the family $\set{U\cap Y:U\in\gamma}$ is locally finite in $Y$.
\end{definition}

For a Tychonoff $X$, $Y$ is $C$-dense in $X$ if and only if $Y$ is dense in $X$ and $C$-embedded in $X$.

\begin{proposition}[\cite{rezn2008,rezn2022-2008}]\label{pbmb4-1+1}
If $X$ is a quasi-regular space, then $Y\subset X\subset \cl Y$.
Let $\ggg\in\sset{\gbm_o,\god_{o,l}}$.
\begin{itemize}
\item[\rm(1)]
 If $Y$ is a $\ggg$-nonmeager ($\ggg$-Baire) space, then $X$ is a $\ggg$-nonmeager ($\ggg$-Baire) space.
\item[\rm(2)]
Let $X$ be a quasi-regular space and $Y$ be $C$-dense in $X$. A space $Y$ is $\ggg$-nonmeager ($\ggg$-Baire) if and only if $Y$ is $\ggg$-nonmeager ($\ggg$-Baire).
\end{itemize}
\end{proposition}
\begin{proof}
(1) Each strategy of player $\al$ in the game on $Y$ is assigned a strategy on $X$. Let $U_0,V_0,\dots $  be the open subsets of $X$ constructed on the $n$th move. According to the strategy on $Y$, the player $\al$ chooses a set $U_n'\subset Y$ open in $Y$ depending on the sets $U_0\cap Y,V_0 \cap Y,\dots $\,. The open set $U_n\subset X$ is chosen in such a way that $U_n \cap Y= U_n'$ and $U_n\subset V_{n-1}$. If $\al$ wins the $Y$ game, then it wins the $X$ game as well.

(2) By virtue of (1), it suffices to show that if $Y$ is $C$-dense in $X$, then if $X$ is a $\ggg$-nonmeager ($\ggg$-Baire) space, then so is $Y$. By Proposition \ref{pbmg5}, $\al$ has a regular strategy on $X$. The regular strategy of player $\al$ in the game on $X$ is associated with the strategy on $Y$. Let open sets $U_0',V_0',\dots $ of the space $Y$ be constructed on the $n$th move. In accordance with the strategy on $X$, player $\al$ chooses a set $U_n\subset \cl{U_n}\subset V_{n-1}$ open in $X$ depending on the sets $U_0=\Int{\cl{U_0'}},V_0=\Int{\cl{V_0'}},\dots$\,. We set $U_n'=U_n\cap Y$. If $\al$ wins the $X$ game, then it wins the $Y$ game as well.
\end{proof}

Propositions \ref{pbmb1-1}, \ref{pbmb2-1}, and \ref{pbmb2} and Theorem \ref{tbmb-bm} imply

\begin{proposition} \label{pbmb4}
Let $X$ be a space.
In the diagrams below, the arrow
\[
A \to B
\]
means that
\begin{itemize}
\item[\rm (1)] if $X$ is an $A$-nonmeager space, then $X$ is a $B$-nonmeager space;
\item[\rm (2)] if $X$ is an $A$-Baire space, then $X$ is a $B$-Baire space;
\item[\rm (3)] if $X$ is an $A$-space, then $X$ is a $B$-space.
\end{itemize}
The bottom arrow means that $A$-nonmeager and $A$-Baire imply nonmeager and Baire.

{
\def\b#1{\gbm_{#1}}
\def\o#1#2{\god_{#1,#2}}
\def\a#1{\arrow[#1]}
\[
\begin{tikzcd}
&\b f\a r & \b k \a{ddr}\a{dr}&
\\
\b o \a{ddr} &&& \b p \a{lll} \a{ddl}
\\
\o ok \a{dr} &&& \o pk \a{lll}\a{dl}
\\
&\o ol \a d & \o pl \a l&
\\
& \gbm &&
\end{tikzcd}
\]}
\end{proposition}

\section{$\gbm_r$ and $\god_{t,q}$ spaces}\label{sec-bms}

In this section, we study  $\gbm_r$ and $\god_{t,q}$ spaces. The relationship between  these spaces is shown by the following statement, which follows from Propositions \ref{pbmb1} and \ref{pbmb3}.

\begin{proposition}\label{pbms1}
Let $X$ be a space, and let $\ggg\in\sset{\gbm_r,\god_{t,q}}$, where $r\in\sset{o,p,f,k}$, $t\in\sset{o,p}$ and $q\in\sset{l,k}$.
\begin{itemize}
\item[\rm(1)]
If $X$ is nonmeager and $X$ is a $\ggg$-space, then $X$ is $\ggg$-nonmeager.
\item[\rm(2)]
If $X$ is a Baire space and $X$ is a $\ggg$-space, then $X$ is a $\ggg$-Baire space.
\end{itemize}
\end{proposition}

\begin{proposition}\label{pbms2}
If $X$ is a space, $\uu,\vu\in\UU(X)$, $r\in\sset{o,p,f,k}$, $t\in\sset{o,p}$ and $q\in\sset{l,k}$, then
\begin{align*}
BM_r^*(X)&\sim BM(X,\cV^*;\uu,\vu),
&
OD_{t,q}^*(X)&\sim OD(X,\pn_t(X),\cW^*;\uu,\vu),
\end{align*}
where $\cV^*= \vv_r(X)\cup \vvbmi(X)$ and $\cW^*= \ww_q(X)\cup \ww_e(X)$. If $\cP$ is a $\pi$-base in $X$, then
 $\uu_p(X,\cP)\in\UU(X)$, and if the space $X$ is quasi-regular, then $\uu_r(X),\uu_{pr}(X,\cP)\in\UU(X )$.
 Moreover, the following assertions hold:
\begin{itemize}
\item[\rm(1)]
$X\in\gbm_r$ if and only if $BM(X,\cV^*;\uu,\vu)$ is $\al$-favorable.
\item[\rm(2)]
$X\in\god_{t,q}$ if and only if $OD(X,\pn_t(X),\cW^*;\uu,\vu)$ is $\al$-favorable.
\item[\rm(3)]
$X\in\god_{o,q}$ if and only if $OD(X,\cP,\cW^*;\uu,\vu)$ is $\al$-favorable if and only if $OD(X,\cP,\cW^*;\wt\uu,\wt\uu)$ is $\al$-favorable, where $\wt\uu=\uu_p(X,\cP)$.
\item[\rm(4)]
if $X\in\gbm_r$, then for any $\nom$ there exists a winning strategy $s$ for player $\al$ such that the following condition is satisfied: if player $\be$ chooses $V_k=X$ at step $k<n$, then $\al$ chooses $U_{k+1}=X$.
\item[\rm(5)]
if $X\in\god_{t,q}$, then for any $\nom$ there exists a winning strategy $s$ for player $\al$ such that the condition is satisfied: if player $\be$ chooses $V_k=X$  at step $k<n$, then $\al$ chooses $U_{k+1}=X$.
\item[\rm(6)]
if $X\in\god_{o,q}$, then there exists a winning strategy $s$ for player $\al$ such that the following condition is satisfied: if player $\be$  chooses $V_n=X$ at step $n$ and $M_n=X$, then $\al$ chooses $U_{n+1}=X$.
\end{itemize}
\end{proposition}
\begin{proof}
The equivalence of games and assertions (1) and (2) follow from Propositions \ref{pbmg1} and \ref{pbmb4-1}. Assertion (3) follows from Proposition \ref{pbmg4}.

Let us prove (4). Let $\ts$ be a winning strategy for $\al$. We define a strategy $s$. Suppose that  $k>0$ and on the first $k$ moves sets $U_0$, $V_0$, $U_1$, \dots , $U_k$, $V_k$ are chosen.
Player $\alpha$ chooses $U_{k+1}$ as prescribed by the strategy $s$. If $V_k=X$ and $k<n$, then $U_{k+1}=X$. Otherwise, 
$U_{k+1}=\ts(U_{n_0},V_{n_0},U_ {n_0+1},\dots ,U_k,V_k)$ for $n_0=\min \set{l,n: l<n, V_l\neq X}$.

Let us prove (5). Let $\ts$ be a winning strategy for $\al$. We define a strategy $s$. Suppose that  $k>0$ and  on the first  $k$ moves sets $U_0$, $V_0$, $M_0$, $U_1$, \dots , $U_k$, $V_k$, $M_k$  are selected. Player $\alpha$ chooses $U_{k+1}$ as prescribed by the strategy $s$. If $V_k=X$ and $k<n$, then $U_{k+1}=X$. Otherwise,  $U_{k+1}=\ts(U_{n_0},V_{n_0},M_ {n_0},U_{n_0+1},\dots ,U_k,V_k,M_k)$ for $n_0=\min \set{l,n: l<n, V_l\neq X}$.

Let us prove (6). Let $\ts$ be a winning strategy for $\al$. We define a strategy $s$. Suppose that $k>0$ and  on the first $k$ moves sets $U_0$, $V_0$, $M_0$, $U_1$, \dots , $U_k$, $V_k$, $M_k$. Player $\alpha$  chooses $U_{k+1}$ as prescribed by the strategy $s$. If $V_k=X$ and $M_k=X$, then $U_{k+1}=X$. Otherwise,  $U_{k+1}=\ts(U_{n_0},V_ {n_0},M_{n_0},U_{n_0+1},\dots ,U_k,V_k,M_k)$ for $n_0=\min \set{l<k: V_l\neq X \text{ or }M_l\neq X}$.
\end{proof}


Let $D$ be an index set, $(X_\de,\tp_\de)$ be a space, and $\tps_\de=\tp_\de\setminus \sset\es$ for $\de\in D$; we set $ X=\prod_{\de\in D}X_\de$. For $\sq{U_\de}{\de\in D}\in\prod_{\de\in D}{\tps_\de}$ we denote $\supp (\sq{U_\de}{\de\in D})=\set{\de\in D: U_\de\neq X_\de}$. The family
\[
\pbase{\sq{X_\de}{\de\in D}}\eqdef \set{\prod_{\de\in D}U_\de: \sq{U_\de}{\de\in D} \in\prod_{\de\in D}{\tps_\de}\text{ and } |\supp (\sq{U_\de}{\de\in D})|<\om}
\]
is a base of the space $X$.

Proposition \ref{pbms2} implies the following assertion.
\begin{proposition}\label{pbms3}
Let $r\in\sset{o,p,f,k}$, $q\in\sset{l,k}$,
$D$ be an index set, $(X_\de,\tp_\de)$ be a space for  each $\de\in D$,  and $X=\prod_{\de\in D}X_\de$. Let $\cB=\pbase{\sq{X_\de}{\de\in D}}$, $\cV^*= \vv_r(X)\cup \vvbmi(X)$, $\cW^* = \ww_q(X)\cup \ww_e(X)$, and $\uu=\uu_p(X,\cB)$.
 Then
\begin{itemize}
\item[\rm(1)]
$X\in\gbm_r$ if and only if $BM(X,\cV^*;\uu,\uu)$ is $\al$-favorable;
\item[\rm(2)]
$X\in\god_{p,q}$ if and only if $OD(X,\pn_p(X),\cW^*;\uu,\uu)$ is $\al$-favorable;
\item[\rm(3)]
$X\in\god_{o,q}$ if and only if $OD(X,\cB,\cW^*;\uu,\uu)$ is $\al$-favorable.
\end{itemize}
\end{proposition}

\begin{proposition}\label{pbms4}
If  $r\in\sset{o,p}$,  and $X\in\gbm_r$, $Y\in \gbm_k$, then $X\times Y\in \gbm_r$.
\end{proposition}
\begin{proof} Let $s_X$ and $s_Y$ be  winning strategies for $\al$ on $X$ and $Y$, respectively.
Let us describe a winning strategy for $\al$. It follows from Proposition \ref{pbms3} that it suffices to consider the case when the player $\be$ chooses sets of the form $V_n=V_{X,n}\times V_{Y,n}\subset X\times Y$. At the $n$th step, we put $U_{X,n}=s_X(U_{X,0},V_{X,0},\dots ,V_{X,n-1})$, $U_ {Y,n}=s_Y(U_{Y,0},V_{Y,0},\dots ,V_{Y,n-1})$, and $U_n=U_{X,n}\times U_{ Y,n}$.
\end{proof}

\begin{proposition}\label{pbms5}
Let $r\in\sset{k,f}$, and let $X_n\in\gbm_r$ for $\nom$. Then $X=\prod_\nom X_n\in \gbm_r$.
\end{proposition}
\begin{proof}
Let us describe a winning strategy for $\al$.
Let $s_n$ be a winning strategy for $\al$ on $X_n$ that satisfies condition (4) of \ref{pbms2}. We set $\cB=\pbase{\sqnn X}$.
It follows from proposition \ref{pbms3} that it suffices to consider the case when the player $\be$ chooses sets of the form $V_k=\prod_\nom V_{n,k}\in\cB$. At the $k$th step, we put $U_{n,k}=s_n(U_{n,0},V_{n,0},\dots ,V_{n,k-1})$ for $\nom$ and $U_k=\prod_\nom U_{n,k}$.
\end{proof}


\begin{assertion} \label{abng1}
Let $(X,\tp)$ be a quasi-regular space, $\tps=\tp\setminus\sset\es$, $\ga_n\subset \tps$, $\cl{\bigcup\ga_n}=X$ for $ \nom$, $\cV\subset \vv(X)$, and $\cV^*=\cV\cup \vvbmi(X)$. Suppose that the following condition is met:
\begin{itemize}
\item
if $\sqnn U\in\vvr(X)$, $U_0=X$ and for each $n>0$ there exists $W_n\in\ga_n$, such that $U_{n}\subset W_n$, then either $\bigcap_{\nom}U_n=\es$ or $\sqnn U \in\cV$.
\end{itemize}
Then $BM(X,\cV^*)$ is $\al$-favorable.
\end{assertion}
\begin{proof}
Let us describe a winning strategy for $\al$. Let $U_0=X$. For $n>0$, at the $n$th step the player $\al$ chooses $U_n\in\tps$ in such a way that the following conditions are satisfied:
\begin{itemize}
\item $\cl{U_n}\subset V_{n-1}$;
\item $\cl{U_n}\subset W_{n}$ for some $W_n\in\ga_n$.
\end{itemize}
\end{proof}


\begin{theorem}\label{tbms1}
\begin{itemize}
\item[\rm (1)] $\gbm_f\subset \gbm_k\subset \gbm_p\subset \gbm_o$.
\item[\rm(2)]
Let $r\in\sset{o,p}$, $X\in\gbm_r$, and $Y\in \gbm_k$. Then $X\times Y\in \gbm_r$.
\item[\rm(3)]
Let $r\in\sset{k,f}$  and $X_n\in\gbm_r$ for $\nom$. Then $\prod_\nom X_n\in \gbm_r$.
\item[\rm(4)]
For $r\in\sset{f,k,p,o}$, if $X\in\gbm_r$ and $U\subset X$ are open subspaces, then $U\in\gbm_r$.
\item[\rm(5)]
For $r\in\sset{f,k,p,o}$, if the space $X$ is locally $\gbm_r$ (that is, any point has a neighborhood $U\in\gbm_r$), then $ X\in\gbm_r$.
\item[\rm(6)]
If $X$ is a quasi-regular space and belongs to one of the classes  $(\gbm_t)$, $t\in\sset{f,k,p,o}$, listed below, then $X$ is a $\gbm_t$-space:
\begin{itemize}
\item[$\mathrm{(\gbm_f)}$] metrizable spaces, Moore spaces, developable space, semiregular $\sigma$-spaces, and semiregular spaces with a countable network;
\item[$\mathrm{(\gbm_k)}$] compact spaces, $p$-spaces, semiregular strongly $\Sigma$-spaces;
\item[$\mathrm{(\gbm_p)}$] countably compact spaces, semiregular $\Sigma$-spaces, $w\D$-spaces;
\item[$\mathrm{(\gbm_o)}$] feebly compact spaces.
\end{itemize}
\end{itemize}
\end{theorem}
\begin{proof}
Item (1) follows from Proposition \ref{pbmb4}, item (2) follows from Proposition \ref{pbms4}, and item (3) follows from Proposition \ref{pbms5}.

Let us prove (4). Player $\al$ has a winning strategy on $X$.

Let us prove (5). After the first move, player $\al$ chooses $U_1\subset V_0$ such that $U_1\in\gbm_r$ and  then follows the winning strategy for $U_1$.

Let us prove (6).
Let $\tp$ be the topology of $X$ and $\tps=\tp\setminus\sset\es$.
For $\cF\subset \Expne X$ we denote
\[
\Om(\cF)=\set{U\in\tps:\text{ either }U\cap M=\es\text{ or }\cl U\subset \cl M\text{ for }M\in\cF}.
\]
If $\cF$ is locally finite, then $\cl{\bigcup \Om(\cF)}=X$.
For $t\in\sset{f,k,p,o}$ and $\cV=\vv_t(X)$ we construct $\sqnn\ga$ as in Assertion \ref{abng1}.
\begin{itemize}
\item[$\mathrm{(\gbm_f)}$]
Let $X$ be a developable space. Take a development $\sqnn\ga$  of the space $X$.

Let $X$ be a semiregular $\sigma$-space. Let $\sqnn\cF$ be a sequence of locally finite families such that $\bigcup_{\nom}\cF_n$ is a network. We set $\ga_n=\Om(\cF_n)$.
\item[$\mathrm{(\gbm_k)}$]
Let $X$ be a compact space. We put $\ga_n=\sset{X}$.

Let $X$ be a $p$-space. We set $\ga_n$  equal to the family $\cU_n$ from the definition of $p$-spaces (Definition 3.15, \cite{gru1984}).

Let $X$ be a strongly $\Sigma$-spaces. We put $\ga_n=\Om(\cF_n)$, where $\cF=\bigcup_n\cF_n$ is a $\sigma$-discrete family in Definition 4.13 of \cite{gru1984}.
\item[$\mathrm{(\gbm_p)}$]
Let $X$ be a countably compact space. We put $\ga_n=\sset{X}$.

Let $X$ be a $\Sigma$-space. We put $\ga_n=\Om(\cF_n)$, where $\cF=\bigcup_n\cF_n$ is a $\sigma$-discrete family in Definition 4.13 of \cite{gru1984}.

Let $X$ be a $w\D$-space. We set $\ga_n$ equal to the family $\cG_n$ in Definition 3.1 of \cite{gru1984}).
\item[$\mathrm{(\gbm_o)}$]
Let $X$ be a feebly compact space. We put $\ga_n=\sset{X}$.
\end{itemize}
\end{proof}


In \cite{arh2010} proved Theorem \ref{tbms1} $(\gbm_k)$ for $p$-spaces.

\begin{proposition}[\cite{rezn2008,rezn2022-2008}]\label{pbms9}
If $t\in\sset{o,p}$,
$X$ is a $\god_{t,k}$-space, and $Y$ is a $\god_{t,l}$-space, then $X\times Y$ is a $\god_{t,l}$-space.
\end{proposition}
\begin{proof}
Open sets of the form $V\times U$, where $V\subset X$ and $U\subset Y$, form a base $\cB$ of the space $X\times Y$.
Let us define a winning strategy for $\al$. By Proposition \ref{pbms2}, it suffices to consider the case when players $\al$ and $\be$ choose open sets of the form $V\times U\in\cB$ and $\be$ chooses sets $M_n$ of the form $ M_n=M_{X,i}\times M_{Y,i}$, $M_{X,i}\in \pn_t(X)$ and $M_{Y,i}\in \pn_t(Y)$.
 On the $n$th move, we choose open nonempty $U_{X,n}\subset V_{X,n-1}$ and $U_{Y,n}\subset V_{Y,n-1}$ according to strategies on $X$ and $Y$, where $V_{n-1}=V_{X,n-1}\times V_{Y,n-1}$ and $M_{n-1}=M_{X,n-1}\times M_{Y,n-1}$ is the choice of $\be$ at the $(n-1)$th step. Let $U_n=U_{X,n}\times U_{Y,n}$. 
Let us check that the player $\al$ won.

Let $\sqnn x\in \ll(\sqn{V_{X,n},M_{X,n}})$.
Let $y\in \ltu_{\nom} M_{Y,n} \cap \bigcap_{\nom} V_{Y,n}$. Let $N(U)=\set{n\in\om: M_{Y,n}\cap U\neq\es}$ for $U\subset Y$ and 
\[
\cF=\set{N(U) :U\text{ is a neighborhood of the point }x}. 
\]
The family $\cF$ is a filter on $\om$. Let $p\in\oms$ be some ultrafilter containing $\cF$. There is $x\in\bigcap_\nom V_{X,n}$ for which $x=\limp p\sqnn x$. Then $(x,y)\in \ltu_{\nom} M_n \cap \bigcap_{\nom} V_n$.
\end{proof}

\begin{proposition}[\cite{rezn2008,rezn2022-2008}]\label{pbms10}
Let $D$ be an index set  and let
$X_\de$ be a $\god_{o,k}$-space for $\de\in D$. Then $X=\prod_{\de\in D}X_\de$ is a $\god_{o,k}$-space.
\end{proposition}
\begin{proof}
By virtue of proposition \ref{pbms3} (3), it suffices to consider the case when players $\al$ and $\be$ choose sets $U_n,V_n,M_n$ from $\cB=\pbase{\sq{X_\de }{\de\in D}}$.

Let us define a winning strategy for player $\al$. Let $s_\de$ be a winning strategy for $\al$ on $X_\de$ satisfying condition (6) in Proposition \ref{pbms2}. Suppose that $n-1$ moves are made and  sets $U_k,V_k,M_k\in\cB$, $U_k=\prod_{\de\in D}U_{\de,k}$, $V_k=\prod_{ \de\in D}V_{\de,k}$, $M_k=\prod_{\de\in D}M_{\de,k}$ for $k<n$  are chosen. We put
\[
U_{\de,n} = s_\de(U_{\de,0},V_{\de,0},M_{\de,0},\dots ,U_{\de,n-1}, V_{\de,n-1},M_{\de,n-1})
\]
for $\de\in D$ and $U_{n} = \prod_{\de\in D} U_{\de,n}$. Since $U_{\de,k}=V_{\de,k}=M_{\de,k}=X_\de$ for almost all $\de$, we have $U_n\in\cB$.
\end{proof}

\begin{proposition}\label{pbms11}
Let $X_n$ be a $\god_{p,k}$-space for $\nom$. Then $X=\prod_{\nom}X_n$ is a $\god_{p,k}$-space.
\end{proposition}
\begin{proof}
Let us describe a winning strategy for $\al$.
Let $s_n$ be a winning strategy for $\al$ on $X_n$ satisfying condition (5) of \ref{pbms2}. We denote $\cB=\pbase{\sqnn X}$.
It follows from Proposition \ref{pbms3} that it suffices to consider the case when the players $\al$ and $\be$ choose the sets $U_k,V_k\in\cB$. Suppose that $U_j,V_j\in\cB$, $x_j\in X$, $U_j=\prod_\nom U_{j,n}$, $V_j=\prod_\nom V_{j,n}$, and  $x_j=\sqn{x_{j,n}}$ for $j<k$. Let $U_{n,k}=s_n(U_{n,0},V_{n,0},x_{n,0},\dots ,V_{n,k-1},x_{n,k -1})$ for $\nom$ and $U_k=\prod_\nom U_{n,k}$.
\end{proof}


\begin{theorem}\label{tbms2}
\begin{itemize}
\item[\rm(1)]
In the diagram below, each arrow $A \to B$
means that $A\subset B$.
{
\def\b#1{\gbm_{#1}}
\def\o#1#2{\god_{#1,#2}}
\def\a#1{\arrow[#1]}
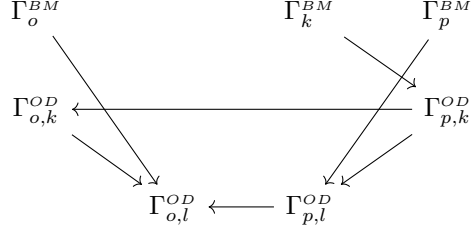
\[
\begin{tikzcd}
\b o \a{ddr} &&\b k \a{dr}& \b p \a{ddl}
\\
\o ok \a{dr} &&& \o pk \a{lll}\a{dl}
\\
&\o ol & \o pl \a l&
\end{tikzcd}
\]}

\item[\rm(2)]
Let $t\in\sset{o,p}$,
$X$ be a $\god_{t,k}$-space and $Y$ be a $\god_{t,l}$-space. Then $X\times Y$ is a $\god_{t,l}$-space.
\item[\rm(3)]
Let $D$ be an index set, and 
$X_\de$ be a $\god_{o,k}$-space for $\de\in D$. Then $\prod_{\de\in D}X_\de$ is a $\god_{o,k}$-space.
\item[\rm(4)]
Let $X_n$ be a  $\god_{p,k}$-space for $\nom$. Then $\prod_{\nom}X_n$ is a $\god_{p,k}$-space.
\item[\rm(5)]
For $t\in\sset{o,p}$ and $q\in\sset{l,k}$,
if $X\in\god_{t,q}$ and $U\subset X$ is an open subspace, then $U\in\god_{t,q}$.
\item[\rm(6)]
For $t\in\sset{o,p}$ and $q\in\sset{l,k}$,
if  $X$ is locally  $\god_{t,q}$ (i.e.,\ any point has a neighborhood $U\in\god_{t,q}$), then $X\in\god_{t, q}$.
\item[\rm(7)]
If $X$ is a quasi-regular space and belongs to one of the classes  $(\god_{t,q})$ for $t\in\sset{o,p}$ and $q\in\sset{l,k }$  listed below, then $X$ is a $\god_{t,q}$-space:
\begin{itemize}
\item[$\mathrm{(\god_{p,k})}$] metrizable spaces, Moore spaces, developable spaces, semiregular $\sigma$-spaces and semiregular spaces with a countable network,
 compact spaces, $p$-spaces, semiregular strongly $\Sigma$-spaces;
\item[$\mathrm{(\god_{p,l})}$] countably compact spaces, semiregular $\Sigma$-spaces, $w\D$-spaces;
\item[$\mathrm{(\god_{o,l})}$] feebly compact spaces.
\end{itemize}
\end{itemize}
\end{theorem}
\begin{proof}
Item (1) follows from Proposition \ref{pbmb4}, item (2) follows from Proposition \ref{pbms9}, item (3) follows from Proposition \ref{pbms10} and item (4) follows from Proposition \ref{pbms11}.

Let us prove (5). Player $\alpha$ follows a winning strategy for $X$.

Let us prove (6). After the first move, player $\al$ chooses $U_1\subset V_0$ such that $U_1\in\gbm_r$, then follows the winning strategy for $U_1$.

Item (7) follows from Theorem \ref{tbms1}.
\end{proof}


\section{Modifications of the Banach--Mazur game with four players}\label{sec-bme}

To formulate and prove the results of this section, it is necessary to define the game and related concepts precisely.

\subsection{General definition of a game.}
The game $\gi$ is defined by the following components:
\begin{itemize}
\item[(P)] $P$, a set of \term{players}\et(players);
\item[($\Ss$)] $\Ss=\iset{S_\al: \alp}$,  an indexed  family of \term{strategies} of players, in which  each player $\al$ has a nonempty set of strategies $S_\al$. A set
\[
\pwr\Ss_P=\prod\Ss
\]
is the \term{strategy space} in the game;
\item[(R)] $R$, a set of \term{plays}\et(plays), a record of the players' moves after they implement their strategies;
\item[($\pi$)] $\pi:\pwr\Ss_P\to R$, the \term{outcome function}, implementation of player strategies during the game, forming a play in the set of plays $R$;
\item[$\mathrm{(\cO)}$] $\cO=\iset{O_\al:\alp}$, the family of \term{outcomes of the game}: $O_\al$ determines the payoff for player $\al$;
\item[($\nu$)] $\nu: R\to \pwr\cO_P$, the \term{payoff function}, which determines the game outcome: $\nu=\diag_{\alp} \nu_\al$, where $\nu_\al=\pi_\al\circ \nu$.
\end{itemize}

The game goes as follows:
\begin{itemize}
\item each player $\alp$ chooses a strategy $s_\al\in S_\al$;
\item players play the game according to their chosen strategies and obtain a play $r=\pi(s)\in R$, where $s=(s_\al)_{\alp}\in \pwr\Ss_P$;
\item the payoff function $\nu$ determines the result of the play $r$: $(v_\al)_{\alp}=\nu(r)\in \pwr\cO_P$, where  $v_\al=\nu_\al(r)$ is the  payoff for player $\al$.
\end{itemize}

We consider games with 
$\cO=\iset{\B_\al:\alp}$, where $\B_\al=\B=\sset{0,1}$, i.e., when $\nu_\al$ is a Boolean function,
$0$ is treated as {\bf false}\et(false) and $1$ as {\bf true}\et(true). The result of a game $\nu_\al(r)$ is interpreted as the payoff of player $\al$: $\al$ wins if $\nu_\al(r)=1$ and $\al$ loses if $\nu_\al (r)=0$. Such games will be called \term{games with a Boolean payoff function}.

A game with a Boolean payoff function is called a \term{zero-sum game} if for any play $r\in R$ there exists a unique player $\alp$ for which $\nu_\al(r)$ equals $1$.
For games with two players, a zero-sum game is a game in which  the first player's gain is the second player's loss and the first player's loss is the second player's gain, i.e., ${\nu_\be}=\lnot \nu_\al$ if $P =\sset{\al,\be}$.

The player $\al$ is called \term{nature} if $\nu_\al$ is identically equal to zero.

A \term{coalition} $K\subset P$ is any set of players. The set $K^c=P\setminus K$ is  the \term{opposite coalition}. A set
\[
\pwr\Ss_K\eqdef\prod_{\al\in K}S_\al
\]
is called the \term{set of coalition strategies} of $K$ and $s=(s_\al)_{\al\in K}\in \pwr\Ss_K$, a \term{coalition strategy} of $K$.

If a game has a Boolean payoff function, then we denote
\[
\nu_K \eqdef \OR_{\al\in K}\nu_\al.
\]
For $r\in R$, $\nu_K(r)=1$ if and only if $\nu_\al(r)=1$ for some $\al\in K$.

A coalition strategy $s\in\pwr\Ss_K$ for a coalition $K$ is called \term{$K$-winning} if $\pi_K(\pi(s\fun t))=1$ for all $t\in \pwr\Ss_{K^c}$. A game is called \term{$K$-favorable} if the coalition $K$ has a $K$-winning strategy. A game is \term{$K$-unfavorable} if there is no $K$-winning strategy.


The following assertion is  checked  directly.
\begin{assertion} \label{abme1} Let $K\subset T\subset P$. If the game $\gi$ is $K$-favorable, then $\gi$ is $T$-favorable. If the game $\gi$ is $T$-unfavorable, then $\gi$ is $K$-unfavorable.
\end{assertion}

For $\alp$, the strategy $s_\al\in S_\al$ is called \term{$\al$-winning} if the strategy $\{\al\to s_\al\}$ of the coalition $\{\al\}$ is $\{\al\}$-winning.
The game is \term{$\al$-favorable} if it is $\sset\al$-favorable, and the game is \term{$\al$-unfavorable} if it is $\sset\al$-unfavorable.

Let $Q=K^c=P\setminus K$, let $s\in\pwr\Ss_K$ be a strategy of $K$, and let $\al\in Q$ be a coalition. We define  games $\gi'=\gi[s]$ and $\gi''=\gi[s,\al]$. The components of the games $\gi'$ and $\gi''$ are the same, the difference is in the payoff function for the player $\al$.
The game $\gi'$ is defined by the following components:
\begin{itemize}
\item[(P)] $Q$, a  set of players;
\item[($\Ss$)] $\iset{S_\al: \al\in Q}$, a family of strategies;
\item[(R)] $R$, a set of games like in the game $\gi$;
\item[($\pi$)] $\pi':\pwr\Ss_Q\to R$, $\pi'(q)=\pi(q\fun s)$ for $q\in \pwr\Ss_Q $, the outcome function;
\item[$\mathrm{(\cO)}$] Boolean game;
\item[($\nu$)] $\nu'',\nu': R\to \pwr\cO_Q$, the payoff functions: $\nu''_\de=\nu'_\de= \nu_\de$ if $\de\neq\al$, $\nu'_\al=\nu_\al$ and $\nu''_\al=\nu_{K\cup \sset \al}= \nu_\al\lor \nu_{K}$.
\end{itemize}

We call a coalition $K$ \term{nature} if each player in the coalition is nature, that is, $\nu_K\equiv 0$.
A coalition $K$ is called \term{dummy} if $K$ is nature and for any coalition $L\subset Q$ it is $L$-favorable if and only if the game is $L\cup K$-favorable.

The following proposition follows from the definitions.

\begin{proposition} \label{pbme1-1} Let $\gi$ be a game with a Boolean payoff function, $P$ be the set of players in the game $\gi$, $K\subset P$ be a coalition, $Q=K\setminus K$ be the opposite coalition, and $s\in\pwr\Ss_K$.
Then the following assertions hold.
\begin{itemize}
\item[(1)]
The game $\gi[s,\al]$ is a zero-sum game for all $\al\in Q$.
\item[(2)]
The game $\gi[s]$ is a zero-sum game if and only if the coalition $K$ is nature. In this case the game $\gi[s]$ is the same as $\gi[s,\al]$ for all $\al\in Q$.
\item[(3)] Let $K$ be nature. A coalition $K$ is a dummy coalition if and only if $\gi[s]\sim \gi[s']$ for any $s'\in \pwr\Ss_K$.
\end{itemize}
\end{proposition}


\subsection{Definition of the game $\ode(X,\pn,\cW;\Om)$}
Let $(X,\tp)$ be a space, and let  $\tps=\tp\setminus\sset\es$.
Suppose that $\pn$ is a $\pi$-net of $X$, $\cW\subset \ww(X)$ and $\Om\subset \tps$.
\myparagraph{Game parameters:} $X,\pn,\cW$ and $\Om$.
\myparagraph{Game set of players:} $P=\sset{\al,\ga,\be,\de}$  (four-players game).
\myparagraph{The $n$th move:}
On the $n$th move, players choose sets
\[
U_n,G_n,\cG_n,V_n,M_n,D_n,\cD_n;
\]
in details:
\begin{center}
\begin{tabular}{ c | l | l}
player & selection & \\ \hline
$\al$ & $U_n$ & $U_n\in \tps$ \\
$\ga$ & $G_n, \cG_n$ & $G_n\in \tps$, $\cG_n\subset\tps$ \\
$\be$ & $V_n,M_n$ & $V_n\in \tps$, $M_n\in \pn$ \\
$\de$ & $D_n, \cD_n$ & $D_n\in \tps$, $\cD_n\subset\tps$ \\
\end{tabular}
\end{center}

For $U\in\tps$ we denote
\begin{align*}
\ppi U=\set{ (V,\cP)\in \tps \times \Expne{\tps}: \cP\text{ is a $\pi$-base } V}.
\end{align*}

On the first move, for $n=0$, the choice of players is
\begin{center}
\begin{tabular}{ c | l | l}
player & choice & choice definition \\ \hline
$\al$ & $U_0$ & $U_0\in \Om$ \\
$\ga$ & $G_0, \cG_0$ & $G_0=U_0$ and $(G_0,\cG_0)\in \ppi{U_0}$ \\
$\be$ & $V_0,M_0$ & $V_0\in\cG_0$, $M_0\in\pn$ and $M_0\subset U_0$ \\
$\de$ & $D_0, \cD_0$ & $(D_0,\cD_0)\in \ppi{V_0}$ \\
\end{tabular}
\end{center}

On the $n$th move, for $n>0$, we determine the choice of players is
\begin{center}
\begin{tabular}{ c | l | l}
player & choice & choice definition \\ \hline
$\al$ & $U_n$ & $U_n\in \cD_{n-1}$ \\
$\ga$ & $G_n, \cG_n$ & $(G_n,\cG_n)\in \ppi{U_n}$ \\
$\be$ & $V_n,M_n$ & $V_n\in\cG_n$, $M_n\in\pn$ and $M_n\subset U_n$ \\
$\de$ & $D_n, \cD_n$ & $(D_n,\cD_n)\in \ppi{V_n}$ \\
\end{tabular}
\end{center}
Note that for $\nom$
\[
U_0=G_0\supset\dots  \supset U_n \supset G_n \supset V_n \supset D_n \supset U_{n+1} \supset\dots 
\ \ \text{ and } \ \
M_n\subset U_n.
\]

\myparagraph{The conditions for players to win:}
Player $\al$ wins if $\sqn{U_n,M_n}\in \cW$. Player $\be$ wins if $\sqn{U_n,M_n}\notin \cW$. Players $\ga$ and $\de$ are nature, which means that they always lose.


\subsection{Definition of the game $\bme(X;\Om)$}
Let $(X,\tp)$ be the space, $\tps=\tp\setminus\sset\es$ and $\Om\subset \tps$.
\myparagraph{Game parameters:} $X$ and $\Om$.
\myparagraph{The set of players in the game:} $P=\sset{\al,\be}$ (two-player game).
\myparagraph{The $n$th move:}
On the $n$th move, players choose sets
\[
U_n,V_n\in\tps.
\]

On the first move, for $n=0$, player $\al$ chooses $U_0\in\Om$, and player $\be$ chooses $V_0\in\tps$, $V_0\subset U_0$.
On the $n$th move, for $n>0$, player $\al$ chooses $U_n\in\tps$, $U_n\subset V_{n-1}$, and player $\be$ chooses $V_n\in \tps$, $V_n\subset U_n$.

\myparagraph{The conditions for players to win:}
Player $\al$ wins if $\bigcap_\nom V_n\neq\es$, otherwise $\be$ wins.

\subsection{Relationship between the $MB$, $BM$ and $\bme$ games }
We denote $\Om_{BM}=\sset X$, and $\Om_{MB}=\tps$.
From the construction  we see that the following assertion holds.
\begin{assertion}\label{abme3}
\begin{align*}
\bme(X;\Om_{BM})&\sim BM(X),
\\
\bme(X;\Om_{MB})&\sim MB(X).
\end{align*}
\end{assertion}

The Banach--Oxtoby Theorem \ref{tbmb-bm} implies  the following proposition.

\begin{proposition} \label{pbme1}
Let $X$ be a space,  and let $\Om\subset \tps$. The game $\bme(X;\Om)$ is $\be$-unfavorable if and only if $U$ is Baire for some $U\in \Om$.
\end{proposition}


\subsection{Properties of the game $\ode(X,\pn,\cW;\Om)$}

We fix a mapping $\lll: \tps\times \Expne{\tps}\to\tp$ for which the following condition is satisfied: if $(U,\cP)\in \tp\times \Expne{\tps}$ and $ V=\lll(U,\cP)$, then
\begin{itemize}
\item
$V=U$ if $U\in\cP$;
\item
$V\in \cP'=\set{W\in\cP: W\subset U}$ if $\cP'\neq\es$;
\item
$V=\es$ otherwise.
\end{itemize}

\myparagraph{Notation of game components of the game $\ode(X,\pn,\cW;\Om)$:}
For $\ka\in P$, we denote by $S_\ka$ the strategy of player $\ka$,  and put $\Ss=\iset{S_\ka:\ka\in P}$. We denote by $\pi$ the outcome function, $\pi = \pwr \Ss_P\to R$, where $R$ is the set of plays.

\begin{assertion}\label{abme4}
Let $(s_\al,s_\ga,s_\de)\in S_\al\times S_\ga\times S_\de$. There is a strategy $q_\al\in S_\al$ such that for any $(q_\ga,q_\be,q_\de)\in S_\ga\times S_\be \times S_\de$ there exists $s_\be\in S_\be$ so for
$s=\iset{s_\ka:\ka\in P}$, $q=\iset{q_\ka:\ka\in P}$,
\begin{align*}
\sqn{\wt U_n,\wt G_n,\wt \cG_n,\wt V_n,\wt M_n,\wt D_n,\wt \cD_n}&=\pi(s),
\\
\sqn{U_n,G_n,\cG_n,V_n,M_n,D_n,\cD_n}&=\pi(q)
\end{align*}
condition is met:
\begin{itemize}
\item[\rm(a)]
for $\nom$ 
\begin{itemize}
\item[\rm(1)]
$\wt U_n\supset \wt G_n \supset U_n \supset G_n \supset V_n \supset D_n \supset \wt V_n \supset \wt D_n \supset \wt U_{n+1} $;
\item[\rm(2)]
$\wt M_n= M_n\subset U_n \subset \wt U_n$;
\item[\rm(3)]
$\wt V_{n+1}\subset V_{n+1} \subset \wt V_n \subset V_n$.
\end{itemize}
\item[\rm (b)] if the family $\cW$ is monolithic and $\sqn{\wt V_n,\wt M_n}\in \cW$, then $\sqn{V_n,M_n}\in \cW$.
\end{itemize}
\end{assertion}

\begin{proof}
Let $q_\al$ and $s_\be$  be strategies for which (a) holds. Item (b) follows from the definition of monolithic and item (a).
On the first move, for $n=0$, the choice of players is:
\begin{center}
\begin{tabular}{ c | l | l}
strategy & choice & definition of choice \\ \hline
 $s_\al$ & $\wt U_0$ & \\
 $s_\ga$ & $\wt G_0, \wt \cG_0$ & \\
 $q_\al$ & $U_0$ & $U_0 =\wt U_0$ \\
$q_\ga$ & $G_0, \cG_0$ & \\
 $q_\be$ & $V_0, M_0$ & \\
 $q_\de$ & $D_0, \cD_0$ & \\
 $s_\be$ & $\wt V_0, \wt M_0$ & $\wt V_0=\lll(D_0,\wt\cG_0)$, $\wt M_0=M_0$\\
 $s_\de$ & $\wt D_0, \wt \cD_0$ & \\
\end{tabular}
\end{center}

On the $n$th move, for $n>0$, the choice of players is:
\begin{center}
\begin{tabular}{ c | l | l}
strategy & choice & definition of choice \\ \hline
 $s_\al$ & $\wt U_n$ & \\
 $s_\ga$ & $\wt G_n, \wt \cG_n$ & \\
 $q_\al$ & $U_n$ & $U_n = \lll(\wt G_n,\cD_{n-1})$ \\
 $q_\ga$ & $G_n, \cG_n$ & \\
 $q_\be$ & $V_n, M_n$ & \\
 $q_\de$ & $D_n, \cD_n$ & \\
 $s_\be$ & $\wt V_n, \wt M_n$ & $\wt V_n=\lll(D_n,\wt\cG_n)$, $\wt M_n=M_n$\\
 $s_\de$ & $\wt D_n, \wt \cD_n$ & \\
\end{tabular}
\end{center}
\end{proof}

\begin{assertion}\label{abme5}
Let $(s_\ga,s_\be,s_\de)\in S_\ga\times S_\be\times S_\de$. There is a strategy $q_\be\in S_\be$ such that for any $(q_\al,q_\ga,q_\de)\in S_\al \times S_\ga \times S_\de$ there exists $s_\al\in S_\al$ so for
$s=\iset{s_\ka:\ka\in P}$, $q=\iset{q_\ka:\ka\in P}$,
\begin{align*}
\sqn{\wt U_n,\wt G_n,\wt \cG_n,\wt V_n,\wt M_n,\wt D_n,\wt \cD_n}&=\pi(s),
\\
\sqn{U_n,G_n,\cG_n,V_n,M_n,D_n,\cD_n}&=\pi(q)
\end{align*}
condition is met:
\begin{itemize}
\item[\rm(a)]
for $\nom$ 
\begin{itemize}
\item[\rm(1)]
$U_n\supset G_n \supset \wt U_n \supset \wt G_n \supset \wt V_n \supset \wt D_n \supset V_n \supset D_n \supset U_{n+1} $;
\item[\rm(2)]
$\wt M_n= M_n\subset \wt U_n \subset U_n$;
\item[\rm(3)]
$V_{n+1}\subset \wt V_{n+1} \subset V_n \subset \wt V_n$;
\end{itemize}
\item[\rm (b)] if the family $\cW$ is monolithic and $\sqn{V_n,M_n}\in \cW$, then $\sqn{\wt V_n,\wt M_n}\in \cW$.
\end{itemize}
\end{assertion}

\begin{proof}
Let $s_\al$ and $q_\be$ be strategies for which (a) holds. Item (b) follows from the definition of monolithic and item (a).
On the first move, for $n=0$, the choice of players is:
\begin{center}
\begin{tabular}{ c | l | l}
strategy & choice & definition of choice \\ \hline
 $q_\al$ & $U_0$ & \\
$q_\ga$ & $G_0, \cG_0$ & \\
$s_\al$ & $\wt U_0$ & $\wt U_0=G_0=U_0$\\
 $s_\ga$ & $\wt G_0, \wt \cG_0$ & \\
 $s_\be$ & $\wt V_0, \wt M_0$ & \\
 $s_\de$ & $\wt D_0, \wt \cD_0$ & \\
 $q_\be$ & $V_0, M_0$ & $V_0=\lll(\wt D_0,\cG_0), M_0=\wt M_0$ \\
 $q_\de$ & $D_0, \cD_0$ & \\
\end{tabular}
\end{center}

On the $n$th move, for $n>0$, the choice of players is:
\begin{center}
\begin{tabular}{ c | l | l}
strategy & choice & definition of choice \\ \hline
 $q_\al$ & $U_n$ & \\
$q_\ga$ & $G_n, \cG_n$ & \\
$s_\al$ & $\wt U_n$ & $\wt U_n=\lll(G_n,\wt \cD_n)$\\
 $s_\ga$ & $\wt G_n, \wt \cG_n$ & \\
 $s_\be$ & $\wt V_n, \wt M_n$ & \\
 $s_\de$ & $\wt D_n, \wt \cD_n$ & \\
 $q_\be$ & $V_n, M_n$ & $V_n=\lll(\wt D_n,\cG_n), M_n=\wt M_n$ \\
 $q_\de$ & $D_n, \cD_n$ & \\
\end{tabular}
\end{center}
\end{proof}


\begin{theorem}\label{tbme1}
Let $(X,\tp)$ be a space, $\tps=\tp\setminus\sset\es$, $\pn$ be a $\pi$-net of $X$, $\Om\subset \tps$ and a family $\cW\subset \ww(X)$ be monolithic. Let $\gi=\ode(X,\pn,\cW;\Om)$.
\begin{itemize}
\item[\rm (1)] The game $\gi$ is  $\al$-favorable if and only if $\gi$ is $\sset{\al,\ga,\de}$-favorable.
\item[\rm (2)] The game $\gi$ is $\be$-favorable if and only if $\gi$ is $\sset{\ga,\be,\de}$-favorable.
\item[\rm (3)] The game $\gi$ is  $\al$-unfavorable if and only if $\gi$ is $\sset{\al,\ga,\de}$-unfavorable.
\item[\rm (4)] The  game $\gi$ is $\be$-unfavorable if and only if $\gi$ is $\sset{\ga,\be,\de}$-unfavorable.
\end{itemize}
\end{theorem}
\begin{proof}
Items (3) and (4) follow from (1) and (2).

Let us prove (1). By Assertion \ref{abme1}, it suffices to show that if the game $\gi$ is $\sset{\al,\ga,\de}$-favorable, then $\gi$ is $\al$-favorable. Let $\{\al\to s_\al, \ga\to s_\ga, \de\to s_\de \}$ be a $\sset{\al,\ga,\de}$-winning strategy. Let $q_\al\in S_\al$ be the strategy from Assertion \ref{abme4}. Let us show that $q_\al$ is a winning strategy for player $\al$. Let $(q_\ga,q_\be,q_\de)\in S_\ga\times S_\be \times S_\de$. Assertion \ref{abme4}\,(b) and the fact that $\cW$ is monolithic imply that there exists $s_\be\in S_\be$ such that  for
$s=\iset{s_\ka:\ka\in P}$, $q=\iset{q_\ka:\ka\in P}$,
\begin{align*}
\sqn{\wt U_n,\wt G_n,\wt \cG_n,\wt V_n,\wt M_n,\wt D_n,\wt \cD_n}&=\pi(s),
\\
\sqn{U_n,G_n,\cG_n,V_n,M_n,D_n,\cD_n}&=\pi(q)
\end{align*}
the following condition is satisfied: if $\sqn{\wt V_n,\wt M_n}\in \cW$, then $\sqn{ V_n, M_n}\in \cW$.
Since $\{\al\to s_\al, \ga\to s_\ga, \de\to s_\de \}$ is a $\sset{\al,\ga,\de}$-winning strategy, we have $\sqn{\wt V_n,\wt M_n}\in \cW$. Hence $\sqn{ V_n, M_n}\in \cW$, and player $\al$ wins with the strategy $q_\al$.

Let us prove (2). By Assertion \ref{abme1}, it suffices to show that if $\gi$ is $\sset{\ga,\be,\de}$-favorable, then $\gi$ is $\be$-favorable. Let $\{\ga\to s_\ga, \be\to s_\be, \de\to s_\de \}$ be a $\sset{\ga,\be,\de}$-winning strategy. Let $q_\be\in S_\be$ be the strategy from Assertion \ref{abme5}. Let us show that $q_\be$ is a winning strategy for player $\be$. Let $(q_\al,q_\ga,q_\de)\in S_\al\times S_\ga \times S_\de$. Assertion \ref{abme5}\,(b) and the fact that $\cW$ is monolithic imply that there exists $s_\al\in S_\al$ such that for
$s=\iset{s_\ka:\ka\in P}$, $q=\iset{q_\ka:\ka\in P}$,
\begin{align*}
\sqn{\wt U_n,\wt G_n,\wt \cG_n,\wt V_n,\wt M_n,\wt D_n,\wt \cD_n}&=\pi(s),
\\
\sqn{U_n,G_n,\cG_n,V_n,M_n,D_n,\cD_n}&=\pi(q)
\end{align*}
the following condition is satisfied: if $\sqn{V_n,M_n}\in \cW$, then $\sqn{\wt V_n,\wt M_n}\in \cW$.
Since $\{\ga\to s_\ga, \be\to s_\be, \de\to s_\de \}$ is a $\sset{\ga,\be,\de}$-winning strategy, we have $\sqn{\wt V_n,\wt M_n}\notin \cW$. Hence $\sqn{ V_n, M_n}\notin \cW$, and player $\be$ wins with the strategy $q_\be$.
\end{proof}
The definition of the game $\ode$ and Theorem \ref{tbme1} imply the following result.
\begin{theorem}\label{tbme1+1}
Let $(X,\tp)$ be a space, $\tps=\tp\setminus\sset\es$, $\pn$ be a $\pi$-net of $X$, $\Om\subset \tps$, $\cW\subset \ww(X)$,  and $K=\sset{\ga,\de}$.
Let $\gi=\ode(X,\pn,\cW;\Om)$.
Then the following assertions hold.
\begin{itemize}
\item[\rm (1)] The coalition $K$ is nature in the game $\gi$.
\item[\rm (2)] If the family $\cW$ is monolithic, then $K$ is a dummy coalition.
\end{itemize}
\end{theorem}

\begin{assertion} \label{abme6-1}
Let $\ts=\{\ga\to\ts_\ga,\de\to\ts_\de\}\in \pwr\Ss_{\sset{\ga,\de}}$ be a strategy of the coalition $\sset {\ga,\de}$, and let $U\in \Om\subset \tps$. If the game $\ode(X,\pn,\cW;\Om_{BM})$ is $\al$-favorable, then there exists a winning strategy $s_\al$ for player $\al$ in the game $\ode(X,\pn,\cW;\Om)$ such that player $\al$ chooses $U_0=U$ on the first move, i.e., $U=s_\al(\es)$.
\end{assertion}

\begin{assertion} \label{abme6}
Let $U\in \Om\subset \tps$. If the game $\ode(X,\pn,\cW;\Om_{BM})$ is $\al$-favorable, then there exists a winning strategy $s_\al$ for player $\al$ in the game $\ode(X,\pn,\cW;\Om)$such that player $\al$ chooses $U_0=U$ on the first move, i.e. $U=s_\al(\es)$.
\end{assertion}
\begin{proof}
Let $\bar s_\al$ be a winning strategy for player $\al$ in the game $\ode(X,\pn,\cW;\Om_{BM})$. Let us define a winning strategy $s_\al$ for player $\al$  in the game $\ode(X,\pn,\cW;\Om)$.
On the first move, player $\al$ chooses $U_0=U$. On the $n$th move, player $\al$ chooses
\[
U_n=\bar s_\al(X,G_0,\cG_0,V_0,D_0,\cD_0,\dots ,U_{n-1},G_{n-1},\cG_{n-1},V_{ n-1},D_{n-1},\cD_{n-1}).
\]
\end{proof}


\begin{assertion} \label{abme7}
Let $(s_\al,s_\ga,s_\be)\in S_\al\times S_\ga\times S_\be$, $U=s_\al(\es)$, and $\pi_{BM} $ be the outcome function in the game $BM(U)$. There is a strategy $q_\be$ of player $\be$ in the game $BM(U)$ such any strategy $q_\al$ of player $\al$ in the game $BM(U)$ there exists a strategy $s_\de \in  S_\de$ such that 
$s=\iset{s_\ka:\ka\in P}$, $q=\iset{q_\ka:\ka\in \sset{\al,\be}}$,
\begin{align*}
\sqn{U_n,G_n,\cG_n,V_n,M_n,D_n,\cD_n}&=\pi(s),
\\
\sqn{\wt U_n,\wt V_n}&=\pi_{BM}(q).
\end{align*}
satisfy the condition
\begin{itemize}
\item[]
$\wt V_n= V_n$ for $\nom$ and $\wt U_0=U$, $\wt U_n=D_{n-1}$ for $n>0$.
\end{itemize}
\end{assertion}
\begin{proof}
Let us define the desired strategies $q_\al$ and $s_\de$.
On the first move, for $n=0$, we define the choice of players  as
\begin{center}
\begin{tabular}{ c | c | l | l}
move &strategy & choice & choice definition \\ \hline
0& $q_\al$ & $\wt U_0=U$ &  \\
0&$s_\al$ & $U_0=U$ & \\
0& $s_\ga$ & $ G_0, \cG_0$ & \\
0& $s_\be$ & $ V_0, M_0$ & \\
0& $q_\be$ & $\wt V_0$ & $\wt V_0=V_0$ \\
1& $q_\al$ & $\wt U_1$ &  \\
0& $s_\de$ & $D_0, \cD_0$ & $D_0=\wt U_1$, $\cD_0=\set{V\in\tps: V\subset D_0}$ \\
\end{tabular}
\end{center}

On the $n$th move, the choice of players is
\begin{center}
\begin{tabular}{ c | c | l | l}
move &strategy & choice & choice definition \\ \hline
n& $s_\al$ & $U_n$ & \\
n& $s_\ga$ & $ G_n, \cG_n$ & \\
n& $s_\be$ & $ V_n, M_n$ & \\
n& $q_\be$ & $\wt V_n$ & $\wt V_n=V_n$\\
n+1& $q_\al$ & $\wt U_{n+1}$ & \\
n& $s_\de$ & $D_n, \cD_n$ & $D_n=\wt U_{n+1}$, $\cD_n=\set{V\in\tps: V\subset D_n}$ \\
\end{tabular}
\end{center}
\end{proof}

\begin{theorem}\label{tbme2}
Let $(X,\tp)$ be a space, $\tps=\tp\setminus\sset\es$, $\pn$ be a $\pi$-net of $X$, $\Om\subset \tps$ and  $\cW\subset \ww(X)$.
Let $\gi=\ode(X,\pn,\cW;\Om)$, $\cW^*= \cW\cup \ww_e(X)$, $\gi^*=\ode(X,\pn,\cW^*;\Om_{BM})$,  and $\gi_{BM}=\bme(X;\Om)$. If $\gi^*$ is $\al$-favorable and $\gi_{BM}$ is $\be$-unfavorable, then $\gi$ is $\sset{\ga,\be}$-unfavorable.
\end{theorem}
\begin{proof} Let $(s_\ga,s_\be)\in S_\ga\times S_\be$. We need to find strategies $(s_\al,s_\de)\in S_\al\times S_\de$ such that for
$s=\iset{s_\ka:\ka\in P}$, and
\begin{align*}
\sqn{ U_n, G_n, \cG_n, V_n, M_n, D_n, \cD_n}&=\pi(s)
\end{align*}
we have 
$\sqn{V_n,M_n}\in\cW$.

Since $\gi_{BM}$ is $\be$-unfavorable, it follows from Proposition \ref{pbme1} that there exists a Baire subspace $U\in\Om$.
Let $\bar s_\al$ be a winning strategy for player $\al$ in the game $\gi^*$. Assertion \ref{abme6} implies that there exists a winning strategy $s_\al$ for player $\al$  in the game $\ode(X,\pn,\cW^*;\Om)$ under which player $\al$ chooses $U_0=U$ on the first move. Let $q_\be$ be the strategy of player $\be$ in the game $BM(U)$ from Assertion \ref{abme7}. Since $U$ is a Baire space, it follows by the Banach--Oxtoby  Theorem \ref{tbmb-bm},  that the game $BM(U)$ is $\be$-unfavorable. Therefore, there is a strategy $q_\al$ of player $\al$ in the game $BM(U)$ such that for
$q=\iset{q_\ka:\ka\in \sset{\al,\be}}$  and 
\begin{align*}
\sqn{\wt U_n,\wt V_n}&=\pi_{BM}(q).
\end{align*}
we have
\[
\bigcap_\nom \wt U_n = \bigcap_\nom \wt V_n \neq \es.
\]
From Assertion \ref{abme7} it follows that  there exists an $s_\de\in S_\de$ such that 
$s=\iset{s_\ka:\ka\in P}$ and
\begin{align*}
\sqn{U_n,G_n,\cG_n,V_n,M_n,D_n,\cD_n}&=\pi(s)
\end{align*}
satisfy the condition
\begin{itemize}
\item[]
$\wt V_n= V_n$ for $\nom$ and $\wt U_0=U$, $\wt U_n=D_{n-1}$ for $n>0$.
\end{itemize}
Hence $\bigcap_\nom V_n\neq \es$ and $\sqn{V_n,M_n}\notin\ww_e(X)$. Since the strategy $s_\al$ is winning for $\al$ in the game $\ode(X,\pn,\cW^*;\Om)$, it follows that $\sqn{V_n,M_n}\in \cW^* =\cW\cup \ww_e(X)$. We obtain $\sqn{V_n,M_n}\in\cW$.
\end{proof}


\subsection{Relationship between the $OD$, $DO$ and $\ode$ games }
Let $\uu,\vu\in\UU(X)$.
We define a strategy $\ts_\ga$ for player $\ga$: at the $n$th step, player $\ga$ chooses $G_n=U_n$ and $\cG_n=\uu(G_n)$.
Let us define a strategy $\ts_\de$ for player $\de$: at the $n$th step player $\de$ chooses $D_n=V_n$ and $\cD_n=\vu(D_n)$.
The strategy $\ts=\{\ga\to\ts_\ga,\de\to\ts_\de\}\in \pwr\Ss_{\sset{\ga,\de}}$ is a strategy of the coalition $\sset {\ga,\de}$.
From the construction of games we see that the following assertion holds.
\begin{assertion}\label{abme8}
\begin{align*}
\ode(X,\pn,\cW;\Om_{BM})[\ts]&\sim OD(X,\pn,\cW;\uu,\vu),
\\
\ode(X,\pn,\cW;\Om_{MB})[\ts]&\sim DO(X,\pn,\cW;\uu,\vu),
\end{align*}
\end{assertion}

Let us define a strategy $\bar s_\ga$ for player $\ga$: at the $n$th step player $\ga$ chooses $G_n=U_n$ and $\cG_n=\set{U\in\tps: U\subset  G_n}$.
Let us define a strategy $\bar s_\de$ for player $\de$: at the $n$th step player $\de$ chooses $D_n=V_n$ and $\cD_n=\set{U\in\tps: U\subset D_n}$.
The strategy $\bar s=\{\ga\to\ts_\ga,\de\to\ts_\de\}\in \pwr\Ss_{\sset{\ga,\de}}$ is a strategy of the coalition $\sset{\ga,\de}$.
From the construction of games we see that the following assertion holds.
\begin{assertion}\label{abme9}
\begin{align*}
\ode(X,\pn,\cW;\Om_{BM})[\bar s]&\sim OD(X,\pn,\cW),
\\
\ode(X,\pn,\cW;\Om_{MB})[\bar s]&\sim DO(X,\pn,\cW),
\end{align*}
\end{assertion}
Assertions \ref{abme8} and \ref{abme9}, Theorem \ref{tbme1+1} and Proposition \ref{pbme1-1} imply the following proposition.
\begin{proposition}[Proposition \ref{pbmg2}] \label{pbme2}
Let $X$ be a space, $\pn$ be a $\pi$-net of $X$, $\cW\subset \ww(X)$, and $\uu,\vu\in\UU(X)$.
If $\cW$ is a monolithic family, then
\begin{align*}
OD(X,\pn,\cW)&\sim OD(X,\pn,\cW;\uu,\vu),
\\
DO(X,\pn,\cW)&\sim DO(X,\pn,\cW;\uu,\vu).
\end{align*}
\end{proposition}

\begin{proposition}[Proposition \ref{pbmb3}] \label{pbme3}
Let $X$ be a space, $\pn$ be a $\pi$-net $X$, and $\cW\subset\ww(X)$ be a monolithic family. Suppose that $X$ is a $\god(\pn,\cW)$-space.  Then the following assertions hold.
\begin{itemize}
\item[\rm(1)]
If $X$ is nonmeager, then $X$ is $\god(\pn,\cW))$-nonmeager.
\item[\rm(2)]
If $X$ is Baire, then $X$ is $\god(\pn,\cW)$-Baire.
\end{itemize}
\end{proposition}
\begin{proof}
Let $\cW^*= \cW\cup \ww_e(X)$. Assertion \ref{abme9}, Theorem \ref{tbme1}, and the fact that $X$ is a $\god(\pn,\cW)$-space imply that the game $\ode(X,\pn,\cW^ *;\Om_{BM})$  is $\al$-favorable.

Let us prove (1). Assertion \ref{abme3} and the Banach--Oxtoby Theorem \ref{tbmb-bm} imply that $X$ is nonmeager if and only if the game $\bme(X;\Om_{MB})$ is $\be$-unfavorable. Theorem \ref{tbme2} implies that the game $\ode(X,\pn,\cW;\Om_{MB})$ is $\sset{\ga,\be}$-unfavorable and, moreover, $\be$-unfavorable. Theorem \ref{tbme1} implies that the game $\ode(X,\pn,\cW;\Om_{MB})$ is $\sset{\ga,\be,\de}$-unfavorable. Assertion \ref{abme9} implies that the game $DO(X,\pn,\cW)$ is $\be$-unfavorable, i.e., $X$ is a $\god(\pn,\cW))$-nonmeager space.

Let us prove (2). It follows from Assertion \ref{abme3} and the Banach-Oxtoby Theorem \ref{tbmb-bm} that $X$ is Baire if and only if the game $\bme(X;\Om_{BM})$ is $\be$-unfavorable. Theorem \ref{tbme2} implies that the game $\ode(X,\pn,\cW;\Om_{BM})$ is $\sset{\ga,\be}$-unfavorable and, moreover, $\be$-unfavorable. Theorem \ref{tbme1} implies that the game $\ode(X,\pn,\cW;\Om_{BM})$ is $\sset{\ga,\be,\de}$-unfavorable. Assertion \ref{abme9} implies that the game $OD(X,\pn,\cW)$ is $\be$-unfavorable, i.e. $X$ is a $\god(\pn,\cW))$-nonmeager space.
\end{proof}

\section{Examples and questions}\label{sec-qe}

In this section, we study how different the introduced classes of spaces are.

The following diagram shows the relationship of the most interesting classes of spaces.

Any arrow $\begin{tikzcd} A \arrow[r] & B \end{tikzcd}$ means that any $A$-Baire space is a $B$-Baire space and  the converse is not true.

The diagram follows from Proposition \ref{pbmb4}. Counterexamples will be constructed below.

{
\def\b#1{\gbm_{#1}}
\def\o#1#2{\god_{#1,#2}}
\def\a#1{\arrow[#1]}
\def\na#1{\arrow[#1,dashrightarrow]}
\def\nna#1{\arrow[#1,tail, dashed]}
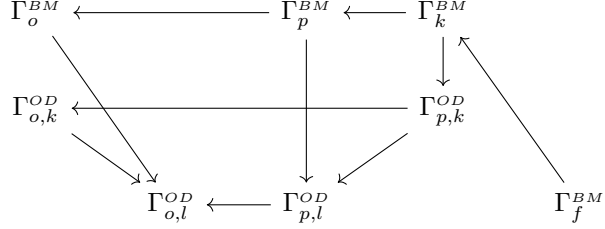
\[
\begin{tikzcd}
\b o \a{ddr} && \b p \a{dd} \a{ll} &\b k \a{d} \a{l} &
\\
\o ok \a{dr} &&& \o pk \a{lll}\a{dl} &
\\
&\o ol & \o pl \a l & & \b f \a{uul}
\end{tikzcd}
\]
}


We denote by $\B$ the discrete two-point space $\sset{0,1}$. The base of the topology in $\B^C$ is formed by sets of the form
\[
W(A,B,C)\eqdef \set{\sq{x_\al}{\al\in C}\in \B^C: x_\al=0 \text{ for }\al\in A\text{ and }x_\be=1 \text{ for }\be\in B }
\]
for finite disjoint $A,B\subset C$.

Propositions \ref{pbms1}, \ref{pbmb4q} and
Theorems \ref{tbms1}, \ref{tbms2}
imply the following assertion.
\begin{assertion} \label{aqe1}
Let $X$ be a regular space without isolated points.
\begin{itemize}
\item[\rm(1)]
If $X$ is compact, then $X$ is $\gbm_k$-Baire.
\item[\rm(2)]
If $X$ is $\god_{o,k}$-nonmeager, then $X$ contains an infinite compact set.
\item[\rm(3)]
If $X$ is countably compact, then $X$ is $\gbm_p$-Baire.
\item[\rm(4)]
If $X$ is $\god_{p,l}$-nonmeager, then $X$ contains a non-discrete countable space.
\item[\rm(5)]
If $X$ is pseudocompact, then $X$ is $\gbm_o$-Baire.
\item[\rm(6)]
If $X$ is $\gbm_f$-nonmeager, then $X$ contains points with a countable base of neighborhoods.
\item[\rm(7)]
If $X$ is $\gbm_o$-nonmeager, then $X$ contains $q_o$-points.
\item[\rm(8)]
If $X$ is a product of locally compact spaces (for example, $X=\R^\tau$), then $X$ is $\god_{p,k}$-Baire.
\end{itemize}
\end{assertion}

Let us give examples that distinguish the classes of spaces under consideration.
In the examples below, \bitem B,A. $X$ means that $X$ is an $A$-Baire space that is not $B$-nonmeager.

\begin{example}\label{eqe1}
\bitem \god_{p,l},\gbm_o.
$X_p$. Indeed, if $X_p$ is an infinite pseudocompact space without isolated points, in which all countable subsets are discrete and closed \cite{Shakhmatov1986,rezn1989}, 
\edemo
then the claim follows from Assertion \ref{aqe1} (5) and (4).
\endedemo
\end{example}

\begin{example}\label{eqe2}
\bitem \god_{o,k},\gbm_p. $X_c$.
If $X_c$ is an infinite countably compact space without isolated points that does not contain infinite compact spaces, for example, $X_c=X\setminus\om$, where $X$ is a countably compact dense subspace of $\be\om$ of cardinality $2^\om$ (see \cite[Proposition 16]{rezn2020}), 
\edemo
then the claim follows from
Assertion \ref{aqe1} (3) and (2).
\endedemo
\end{example}

\begin{example}\label{eqe3}
\bitem\gbm_f,\gbm_k.
$\B^{\om_1}$.
\edemo
The space $\B^{\om_1}$ is a compact space without points of countable character (see Assertion \ref{aqe1} (1) and (6)).
\endedemo
\end{example}

\begin{example}\label{eqe5}
\bitem \god_{p,k},\god_{o,k}. $Y$.
Consider
\begin{align*}
Y_0&=\set{\sq{x_\al}{\al<\om_1}\in\B^{\om_1}: |\set{\al<\om_1: x_\al=0}|\leq \om },
\\
Y_1&=\set{\sq{x_\al}{\al<\om_1}\in\B^{\om_1}: |\set{\al<\om_1: x_\al=1}|< \om} ,
\\
Y&=Y_0\cup Y_1.
\end{align*}
\edemo
Let us show that $OD_{o,k}(Y)$ is $\al$-favorable. A winning strategy for $\al$ is as follows. We choose $U_n$ such that $\cl{U_n}\subset V_{n-1}$. Let $x_n\in M_n\cap Y_0$. Then $K=\cl{\set{x_n:n<\om}}$ is compact.

Let us show that $DO_{p,k}(Y)$ is $\be$-favorable. A winning strategy for $\be$ is as follows. Choose $V_n=W(A_n,B_n,\om_1)$ such that $A_n\subset A_{n-1}$, $B_n\subset B_{n-1}$ and $|A_n|\geq n$, $x_n\in V_n \cap Y_1$, $M_n=\sset{x_n}$. Then $\sqnn x$ is a discrete and closed sequence in $Y$.
\endedemo
\end{example}


\begin{problem} \label{pqe2}
Let $\ggg\in \set{\gbm_r: r\in\sset{f,k,p,o}}\cup \set{\god_{t,q}: t\in\sset{o,p} \text{ and }q\in\sset{l,k}}$.
\begin{itemize}
\item[{\rm (1)}]
Does there exist a $\ggg$-Baire space that is not a $\ggg$-space?
\item[{\rm (2)}]
Is the class of $\ggg$-spaces multiplicative?  That is it true that if $X,Y\in\ggg$, then $X\times Y\in\ggg$?
\item[{\rm (3)}]
Let $X$ and $Y$ be $\ggg$-Baire spaces, and let $X\times Y$ be a Baire space. Is it true that $X\times Y$ is a $\ggg$-Baire space?
\end{itemize}
\end{problem}

The smallest class of spaces among those listed above is the class of $\gbm_f$-spaces, and the largest one is  the class of $\god_{o,l}$-spaces.

\begin{problem} \label{pqe4}
Let $X$ be a regular $\gbm_f$-Baire space.
\begin{itemize}
\item[{\rm (1)}]
Is it true that $X$ is a $\gbm_f$-space and contains a dense metrizable Baire subspace?
\item[{\rm (2)}]
Is it true that $X$ is a $\god_{o,l}$-space?
\end{itemize}
\end{problem}

A space $X$ is called \term{weakly pseudocompact} if there exists a compact Hausdorff extension $bX$ of the space $X$ in which the space $X$ is $G_\de$-dense, i.e., $X$ intersects any nonempty $G_\de$ subset of $bX$ \cite{arh-rezn2005}. It is clear that the product of weakly pseudocompact spaces is weakly pseudocompact; in particular, the product of pseudocompact spaces is weakly pseudocompact.

The next question is a version of Problem \ref{pqe2} (2) and (3).

\begin{problem} \label{pqe5}
Let $X$ be a weakly pseudocompact space (a product of pseudocompact spaces). Which of the following classes does $X$ belong to:
\begin{itemize}
\item[]
$\gbm_o$-spaces, $\gbm_o$-Baire spaces, $\god_{o,l}$-spaces, $\god_{o,l}$-Baire spaces?
\end{itemize}
\end{problem}

\begin{problem} \label{pqe7}
Let $X$ and $Y$ be (completely) regular countably compact spaces. Which of the following classes does the product $X\times Y$ belong to:
\begin{itemize}
\item[]
$\gbm_p$-spaces, $\gbm_p$-Baire spaces, $\god_{p,l}$-spaces, $\god_{p,l}$-Baire spaces, $\gbm_o$-spaces, $\gbm_o $-Baire spaces, $\god_{o,l}$-spaces, $\god_{o,l}$-Baire spaces?
\end{itemize}
\end{problem}

Recall that a group with a topology in which multiplication is continuous is called a \term{paratopological} group.
A group with a topology in which multiplication is separately continuous is called a \term{semitopological} group.
In \cite[Theorem 2.6]{Reznichenko1994} it is proved that a pseudocompact paratopological group is a topological group.
Every weakly pseudocompact semitopological group G of countable $\pi$-character is a topological group metrizable by a complete metric (see \cite[Corollary 2.28]{arh-rezn2005}).

\begin{problem} \label{pqe6}
Let $G$ be a weakly pseudocompact (a product of pseudocompact spaces, a product of two countably compact spaces) paratopological group. Is it true that $G$ is a topological group?
\end{problem}

Note that \cite[Theorem 14]{rezn2022-1} implies that a paratopological $\god_{o,l}$-Baire group is a topological group.


\bibliographystyle{elsarticle-num}
\bibliography{gbtg}
\end{document}